\documentclass[11pt,a4paper]{amsart}
\usepackage{amssymb, amsmath, amsthm}

\usepackage{pstricks}
\usepackage{graphicx}
\usepackage{mathrsfs}
\usepackage{extarrows}
\usepackage{amsthm}
\usepackage{amssymb}
\usepackage{xypic,enumerate, comment, enumitem}
\usepackage[all,cmtip]{xy}
\usepackage{fancyvrb}
\usepackage{longtable}
\usepackage{array}
\newcolumntype{L}{>{$}l<{$}} % math-mode version of "l" column type
\allowdisplaybreaks

\usepackage{hyperref}
\hypersetup{
    colorlinks=true,       % false: boxed links; true: colored links
    linkcolor=blue,        % color of internal links
    citecolor=blue,        % color of links to bibliography
    filecolor=blue,        % color of file links
    urlcolor=blue          % color of external links
}

\theoremstyle{plain}

\newtheorem{theorem}{Theorem}[section]
\newtheorem{proposition}[theorem]{Proposition}
\newtheorem{corollary}[theorem]{Corollary}
\newtheorem{lemma}[theorem]{Lemma}
\newtheorem{conjecture}[theorem]{Conjecture}
\newtheorem*{theorem*}{Theorem}
\newtheorem*{proposition*}{Proposition}
\newtheorem*{corollary*}{Corollary}
\newtheorem*{lemma*}{Lemma}
\newtheorem*{conjecture*}{Conjecture}
\newtheorem*{thmA}{Theorem A}
\newtheorem*{thmB}{Theorem B}
\newtheorem*{thmC}{Theorem C}

\theoremstyle{definition}
\newtheorem{definition}[theorem]{Definition}
\newtheorem{example}[theorem]{Example}

\newtheorem*{definition*}{Definition}
\newtheorem*{example*}{Example}

\theoremstyle{remark}
\newtheorem{remark}[theorem]{Remark}
\newtheorem*{remark*}{Remark}

\newcommand{\HH}{\operatorname{H}}

\newcommand{\p}{\pi}

\newcommand{\G}{\operatorname{G}}

\newcommand{\bQ}{\mathbb{Q}}

\newcommand{\bP}{\mathbb{P}}

\newcommand{\cC}{\mathcal{C}}
\newcommand{\cO}{\mathcal{O}}
\newcommand{\cE}{\mathcal{E}}

\newcommand{\cI}{\mathcal{I}}
\newcommand{\cF}{\mathcal{F}}
\newcommand{\cS}{\mathcal{S}}
\renewcommand{\cH}{\mathcal{H}}
\newcommand{\cV}{\mathcal{V}}
\newcommand{\cG}{\mathcal{G}}
\newcommand{\cA}{\mathcal{A}}
\newcommand{\cM}{\mathcal{M}}
\newcommand{\cU}{\mathcal{U}}

\renewcommand{\cR}{\mathcal{R}}
\newcommand{\cX}{\mathcal{X}}

\newcommand{\cQ}{\mathcal{Q}}

\newcommand{\tV}{\tilde{V}}
\newcommand{\tF}{\tilde{F}}
\newcommand{\tC}{\tilde{C}}

\newcommand{\Pic}{\operatorname{Pic}}

\newcommand{\pr}{\operatorname{pr}}

\newcommand{\id}{\operatorname{id}}

\newcommand{\Bl}{\operatorname{Bl}}

\renewcommand\bar\overline
\renewcommand\tilde\widetilde

\newcommand{\Sym}{\operatorname{Sym}}
\renewcommand{\geq}{\geqslant}
\renewcommand{\leq}{\leqslant}

\newcommand{\PP}{\bP}
\newcommand{\OO}{\mathcal{O}}

\newcommand{\QQ}{\mathbf{Q}}
\newcommand{\ZZ}{\mathbf{Z}}
\newcommand{\CC}{\mathbf{C}}
\newcommand{\VV}{\mathbf{V}}

\newcommand{\PGL}{\operatorname{PGL}}

\newcommand{\CH}{\operatorname{CH}}

\title{Geometry of lines on a cubic fourfold}
\author{Frank Gounelas}
\address{Georg-August-Universit\"at G\"ottingen, Fakult\"at f\"ur Mathematik und Informatik, Bunsenstr. 3-5, 37073
G\"ottingen, Germany}
\email{gounelas@mathematik.uni-goettingen.de}

\author{Alexis Kouvidakis}
\address{Dept. of Mathematics and Applied Mathematics, University of Crete, 70013 Heraklion, Greece.}
\email{kouvid@uoc.gr}

\date{\today}

\begin{document}

\maketitle
\begin{abstract}
For a general cubic fourfold $X\subset\PP^5$ with Fano scheme of lines $F$, we prove a number of properties of the
universal family of lines $I\to F$ and various subloci. We first describe the moduli and ramification
theory of the genus four fibration $p:I\to X$ and explore its relation to a birational model of $F$ in $I$. The main
part of the paper is devoted to describing the locus $V\subset F$ of triple lines, i.e., the fixed locus of the Voisin
map $\phi:F\dashrightarrow F$, in particular proving it is an irreducible projective singular surface of class
$21\mathrm{c}_2(\cU_F)$ and detailing its intersection with the locus $S$ of second type lines. A consequence of the
analysis of the singularities of $V$ is a geometric proof of the fact that if $X$ is very general, then the number of
singular (necessarily 1-nodal) rational curves in $F$ of primitive class is 3780.
\end{abstract}

\setcounter{tocdepth}{1}
\tableofcontents

\section{Introduction}\label{sec:introduction}

Let $X\subset\PP^5_{\CC}$ be a smooth cubic fourfold and $F\subset \G(2,6)$ its Fano scheme of lines. The geometry of
these varieties has received a lot of attention since the beginning of the 20th century and in particular since the
foundational papers \cite{cg, ak}. Consider the universal family $I=\PP(\cU_F)$ of lines, for $\cU_F$ the universal
subbundle of the Grassmannian $\G(2,6)$ restricted to $F$, with the two projections
\[
\xymatrix{
    I\ar[d]^q\ar[r]^p & X \\
    F. &
}
\]
It turns out that the morphism $p$ is a fibration of $(2,3)$-complete intersections in $\PP^3$ associated to the
geometry of $X$, and all but finitely many fibres are genus 4 curves. As such, each fibre $C_x=p^{-1}(x)\subset I$,
which parametrises lines in $X$ through $x$, comes equipped with two $g^1_3$ linear systems by restricting the rulings
of the corresponding quadric. The aim of this paper is to study the fibrations $p,q$ and the geometry and intersection
theory of various geometric loci in each variety.

Recall that there are two types of points $[\ell]\in F$, i.e., lines $\ell\subset X$, depending on the decomposition of
the normal bundle $N_{\ell/X}$. The generic line is called of \textit{first type}, whereas there is a surface $S\subset
F$ parametrising those of \textit{second type} (see Section \ref{sec:background} for formal definitions).

Our first goal is to further expand in Section \ref{sec:curve of lines} on the geometry of the
curves $C_x$, which we achieve by studying properties of the fibration $p$ on $I$ but also the induced restricted
map over $q^{-1}(S)\subset I$. To name one example, we obtain in Corollary \ref{cor:general} that the fibration $p$ is a
fibration in general genus 4 curves.

Moving now to geometric constructions on $F$, Voisin \cite{voisinintrinsic} defined a map
\[\phi: F\dashrightarrow F\]
taking a point $[\ell]$ corresponding to a line $\ell\subset X$ to the \textit{residual line} $[\ell']$, 
in the sense that, if $\ell$ is general, there is a unique $\Pi_\ell=\PP^2$ so that $X\cap\Pi_\ell=2\ell+\ell'$. This
map has been studied in detail in \cite{amerik} and if $X$ does not contain any planes (the generic case, which we are
in), is resolved by a single blowup of the locus $S$. This gives impetus to the construction of the following map
\begin{align*}
\psi:&F\dashrightarrow X \\
&[\ell]\mapsto \ell\cap\ell'.
\end{align*} 
The locus of indeterminacy of $\psi$ also includes the locus $V\subset F$ of \textit{triple lines}, i.e., lines for
which there is a $\Pi_\ell=\PP^2$ so that $X\cap\Pi_\ell=3\ell$. To the best of our knowledge, this locus has so far
eluded study. Note that it can be viewed as the closure of the fixed locus
\[\operatorname{Fix}(\phi)\subset F\] 
of the Voisin map $\phi$. We will prove that it consists of both first and second type lines. Using results of Section
\ref{sec:curve of lines}, we may interpret a general point $[\ell]\in V$ as a triple ramification point of one of the
$g^1_3$'s on $C_x$ for $x\in\ell$. 

The main part of this paper is devoted, in Section \ref{sec:surface V}, to the study of various universal families
related to $V, S$ and $F$. We conclude the following, where $\cU_F$ is the restriction of the universal bundle from the
Grassmannian and $H_F=\mathrm{c}_1(\cU_F^{\vee})$ the restriction of the Pl\"ucker polarisation.

\begin{thmA}
If $X\subset\PP^5$ is a general cubic, then the locus $V\subset F$ of triple lines is a projective irreducible
general type singular surface, whereas the locus $C=S\cap V$ is an irreducible nodal curve. Their classes are
\begin{align*}
[V]&=21{\mathrm c}_2(\cU_F) \in \CH^2(F),\\
[C]&=\frac{35}{2}H_F^3 \in \HH_2(F,\ZZ).
\end{align*}
\end{thmA}

In particular, the class of $V$ is a multiple of the generator (cf.\ \cite{voisinconiveau}) of the known extremal ray of the
effective cone $\operatorname{Pseff}^2(F)$. Much of Section \ref{sec:surface V} is devoted to studying the strict
transform $\tV$ of $V$ in $\Bl_S(F)$, which we prove is the vanishing locus of a section of a tautological bundle on the
product of Grassmannians and is a resolution of singularities of $V$. As a corollary to the analysis of the
singularities of $V$, we obtain the following, which to the best of our knowledge was not previously known (despite
similar results, e.g., \cite[Theorem 1.3]{no}).

\begin{thmB}
    If $X\subset\PP^5_\CC$ is a very general cubic, then the number of 
    \begin{enumerate}
        \item nodal rational curves in $F$ of primitive class $\beta\in\HH_2(F,\ZZ)$,
        \item singularities of $V$, all of which are isolated non-normal with degree two normalisation,
        \item nodes of $C=S\cap V$,
    \end{enumerate}
    are all equal to 3780.
\end{thmB}

As a further application of the analysis above, in Section \ref{sec:ramloci} we study the loci \[R, R', N\subset I\]
which are the closure of the locus of ramification, residual to all ramification and triple ramification points
respectively, of the two $g^1_3$'s on the smooth fibres $C_x$. The projection $R'\to F$ is generically finite of degree
16, the degree of the Voisin map $\phi$. On the other hand, we compute that the map $\psi$ from above has degree 24, and
in Corollary \ref{cor:resolvepsitilde} we prove that it is resolved by
blowing up $S$ and then the strict transform of $V$. 

Finally, putting everything together, we prove the following in Theorem \ref{thm:Bl=R} and compute in Proposition
\ref{prop:classR} the classes of the above loci in the Chow group of $I$ in terms of the tautological line bundle
$l=\OO_I(1)=p^*\OO_X(1)$ and the pullback $q^*H_F$ of the class of the Pl\"ucker polarisation (or the relative canonical
divisor $\omega_p$ of the fibration $p:I\to X$).

\begin{thmC}
If $X\subset\PP^5$ is a general cubic fourfold, then the locus of ramification points $R$ is irreducible and smooth outside $q^{-1}(S)\subset R$ and has
normalisation the blowup of $F$ first at $S$ and then the strict transform of $V$. The classes of $R, R'$ in $\CH^1(I)$
and $N$ in $\CH^2(I)$ are given as follows
\begin{align*}
%    [R]&=4g+l,\\
%    [R']&=4g+16l,\\
%    [N]&= 4l^2 -4lg + 25c.
    [R]&=4q^*H_F+l=4\mathrm{c}_1(\omega_p)-3l,\\
    [R']&=4q^*H_F+16l=4\mathrm{c}_1(\omega_p)+12l,\\
    [N]&=21(l^2-lq^*H_F)=21(2l^2-l\mathrm{c}_1(\omega_p)).
\end{align*}
\end{thmC}

In the appendix, we reprove the first two of the above class computations using the genus four fibration $I\to X$ and
some intersection theory on the moduli space of admissible covers.

\textbf{Acknowledgements} We would like to thank Olof Bergvall and Daniel Huybrechts for helpful correspondence. The
first author was supported by the ERC Consolidator Grant 681838 ``K3CRYSTAL''.

\section{Background and Notation}\label{sec:background} 

As the notation surrounding cubic fourfolds is substantial, we devote this section to fixing that used in the
paper and recalling some basic properties, so that it acts as a reference for later sections.

For a vector bundle $E$ we denote by $\PP(E)=\mathrm{Proj}(\mathrm{Sym}(E^\vee))$, so that projective space parametrises
one dimensional subspaces. We denote by $\G(k,n)$ the space of $k$-dimensional subspaces of $\CC^n$, with universal
bundle $\cU$ of rank $k$ and universal quotient bundle $\cQ$ of rank $n-k$. When multiple such bundles are in question,
we will denote then by $\cU_k$. Finally, we denote by $\sigma_I$ the standard Schubert cycles for an index $I$ so that,
e.g., $\sigma_i={\mathrm c}_i(\cQ)$ for $i\geq1$.

Throughout, $X\subset\PP^5$ will be a smooth cubic fourfold with $H_X=\OO_{X}(1)$ and $F\subset \G(2,6)$ the Fano
scheme of lines contained in $X$ which is a hyperk\"ahler fourfold \cite{bd}. To unburden notation, we will often be
sloppy in distinguishing a line $\ell\subset X$ and the point $[\ell]\in F$ that it defines. We denote by $\cU_F,\cQ_F$
the restrictions of $\cU,\cQ$ to $F$. The universal family of lines sits in a diagram
\[\xymatrix{
    \PP(\cU)\ar[r]^p\ar[d]^q & \PP^5\\
    \G(2,6)
}\]
and note that we have 
\[p^*\OO_{\PP^5}(1)\cong \OO_{\PP(\cU)}(1).\]
We use the same notation $p, q$ for the induced maps on
\[I:=\PP(\cU_F)\]
the universal family of lines on $F$.

The subvariety $F\subset \G(2,6)$ is given by a section of the rank four bundle $\Sym^3\cU^*\cong q_*p^*\OO_{\PP^5}(3)$
- in fact it is the section induced, under this isomorphism, by $f\in k[x_0,\ldots,x_5]_3$ whose vanishing is $X$ (see
 \cite[Proposition 6.4]{3264}) 
  so its cohomology class in the Grassmannian is given by ${\mathrm c}_4(\Sym^3\cU^*)$
 which can be computed as follows (see \cite[Example 14.7.13]{fulton})
\begin{align}\label{eq:class of F}
\begin{split}
   [F] &= 18{\mathrm c}_1(\cU^*)^2{\mathrm c}_2(\cU^*) + 9{\mathrm c}_2(\cU^*)^2 \\
        &= 18\sigma_1^2\sigma_{1,1}+9\sigma_{1,1}^2 \\
        &= 27\sigma_2^2-9\sigma_1\sigma_3-18\sigma_4.
\end{split}
 \end{align}

Consider now the morphism $p: I \to X$. We denote by 
\[C_x:=p^{-1}(x)\] 
the fibre over $x$, which parametrises the lines in $X$ containing $x$. If $X$ is general, then $C_x$ is 1-dimensional
for all $x$ (see \cite[Lemma 2.5]{lsv}), whereas for arbitrary $X$ there are only finitely many points in $X$ where the
fibre can be two dimensional (see \cite[Corollary 2.2]{coskunstarr}). We defer to Proposition \ref{prop:properties
C_x} for more properties of $C_x$. In any case, $C_x$ embeds in $F$ via $q$ and can also be realised as the $(2,3)$-complete
intersection in $\PP^3$ formed by the intersection points of the lines through $x$ with $T_xX \cap A$, where $A$ a
hyperplane not containing $x$: for $R_xX=\VV(\sum x_i\partial_i f)$ the polar quadric, $C_x$ is the intersection
$R_xX\cap T_xX\cap X\cap A$. As such, if it is 1-dimensional, it is of arithmetic genus 4 and has two $g^1_3$'s,
counted with multiplicities \cite[Corollary D.11]{3264}, namely the restrictions of the rulings of the quadric
$R_xX\cap T_xX\cap A$.

If $X$ is very general, $\Pic(F)=\ZZ$ and we denote by $\beta$ the generator of $\HH_2(F,\ZZ)^{\mathrm alg}$. In
\cite{amerik} it is proven that 
\[[C_x]=2\beta\in\HH_2(F,\ZZ).\]

Following \cite{cg}, there are two types of lines $\ell\in F$, depending on the decomposition of the normal bundle
$N_{\ell/X}$.

\begin{definition}
    We say that a line $\ell\subset X$ is 
    \begin{enumerate}
        \item of \textit{first type} if $N_{\ell/X}\cong\OO(1)\oplus\OO^2$,
        \item of \textit{second type} if $N_{\ell/X}\cong\OO(1)^2\oplus\OO(-1)$.
    \end{enumerate}
\end{definition}
An equivalent geometric description is as follows: $\ell$ is of
\begin{enumerate}
    \item first type if there is a unique $\Pi_\ell=\PP^2$ tangent to $X$ along $\ell$.
    \item second type if there is a unique $H_\ell=\PP^3$ tangent to $X$ along $\ell$ (equivalently, a family $\Pi_{\ell,t}=\PP^2$, $t\in\PP^1$, of $2$-planes tangent to $X$ along $\ell$).
\end{enumerate}
Denote by 
\begin{align*}
S&:=\{\ell : \ell\text{ is of second type}\}\subset F, \\
W&:=p(\PP(\cU_S))=p(q^{-1}(S))\subset X,
\end{align*}
the \textit{locus of second type lines} and the locus of points through which there passes a second type line
respectively. The universal family of second type lines $q^{-1}(S)=\PP(\cU_S)$ is the singular locus
$\operatorname{Sing}(p)$ of the morphism $p:I\to X$ from Lemma \ref{lem:sings Cx}. 

\begin{lemma}(\cite[Corollary 2.2.14]{huybrechts})\label{lem:wirreddiv}
    Let $X\subset\PP^5$ be a smooth cubic fourfold. Then $W$ is 3-dimensional, and if $X$ is general it is irreducible.
\end{lemma}

Denote by $H_F={\mathrm c}_1(\cU^\vee_F)$ the Pl\"ucker ample line bundle on $F$ and by $H_S$ the restriction on $S$.
The following is a combination of \cite[Lemma 1]{amerik}, \cite[\S 3]{osy}, \cite[Proposition 6.4.9]{huybrechts}.
\begin{theorem} \label{thm:amerik and osy}
    If $X\subset\PP^5$ is a cubic fourfold then $S$ is 2-dimensional and is the degeneracy locus of the Gauss map, i.e.,
    the following morphism of vector bundles
    \[ \Sym^2\cU_F \to \cQ_F^\vee.\]
    In particular ${\mathrm c}_1(K_S)=3H_S$ in $\HH^2(S,\bQ)$ and the class of $S$ in $\mathrm{CH}^2(F)$ is given by
    \[[S]=5({\mathrm c}_1(\cU_F^\vee)^2-{\mathrm c}_2(\cU_F^\vee))=5{\mathrm c}_2(\cQ_F)=5\sigma_2|_F.\]
    If $X$ is general, $S$ is a smooth projective irreducible surface.
\end{theorem}

Let 
\begin{align}
\label{diag:voisinmap}
\begin{split}
\phi:F&\dashrightarrow F\\
\ell&\mapsto\ell'
\end{split}
\end{align}
be \textit{the Voisin map} of \cite{voisinintrinsic}, taking a general line $\ell$ and giving the residual line $\ell'$ in
the tangent 2-plane $\Pi_\ell$ to $\ell$, i.e., $\Pi_\ell\cap X=2\ell+\ell'$. Note that this is not defined on $S$ nor
on any lines contained in a plane contained inside $X$. Containing a plane is a divisorial condition in the moduli
space, so for $X$ outside this locus, we can resolve this map with one blowup $\operatorname{Bl}_SF$ along the surface
$S$.  We denote by $E_S$ the exceptional divisor which is a $\PP^1$ fibration over $S$
\begin{equation}\label{diag:voisindiagram}
\begin{split}
\xymatrix{
    & E_S\ar@{^{(}->}[r]\ar[dl] & \operatorname{Bl}_SF\ar[dl]_\pi\ar[dr]^{\tilde{\phi}} & \\
    S\ar@{^{(}->}[r] & F\ar@{-->}[rr]^{\phi}& & F.
}
\end{split}
\end{equation}

As it will be relevant for the count of nodal rational curves later, we record the following.

\begin{theorem}(\cite[Theorem 0.2]{osy})\label{thm:osy}
    If $X$ is a very general cubic fourfold, then for every rational curve $C\subset F$ of class $\beta$ there exists a
    unique $s\in S$ so that $C=\phi(q^{-1}(s))$.
\end{theorem}

One can also define the following rational map 
\begin{align}\label{diag:defpsi}
\begin{split}
\psi: F &\dashrightarrow X\\ 
[\ell] & \mapsto \ell\cap \ell'
\end{split}
\end{align}
taking a general line to the intersection point with its residual in the unique tangent 2-plane. This is not defined on $S$ but also on
the following locus. 
\begin{definition}\label{triplelinelocus}
    Denote by $V\subset F$ the locus of points $[\ell]$ so that there is a 2-plane $\Pi_\ell$ so that $\Pi_\ell\cap
    X=3\ell$. 
\end{definition}
This locus appears briefly in the literature (\cite[Definition 20.4]{svfourier} and also \cite[Lemma 10.15]{cg} where in
the case of a cubic threefold it is proven to be a finite set) and can also be viewed as the set of lines fixed by the
Voisin map $\phi$, that is those $[\ell]\in F$ with $\phi([\ell])=[\ell]$.

We will need the follow results from intersection theory, which we put together for later use. 
In the following we denote by $l:=\OO_{I/F}(1)$ the tautological bundle, $\omega_p:=\det(\Omega^1_p)$ the determinant
of the relative cotangent bundle, and by $\lambda:=\det(p_*\omega_p)$ the determinant of the Hodge bundle of the 
fibration $p$.
\begin{lemma}\label{lem:intersection theory}
Let $X\subset\PP^5$ be a smooth cubic fourfold. We have the following equalities in the Chow ring (or in cohomology)
    \begin{enumerate}
    \item $H_F^4=108$, $H_F^2\, {\mathrm c}_2(\cU_F)=45$, ${\mathrm c}_2^2(\cU_F)=27$,
    \item $H_F^2\, \sigma_2=63$, $\sigma_2^2|_F= 45$,
    \item $\beta H_F=3$, $\beta = {1\over 36}H_F^3$ and so $q_*[C_x]= {1\over 18}H_F^3$,
    \item $q_*p^*H_X^2= H_F$, $q_*p^*H_X^3=\sigma_2|_F$, $q_*p^*H_X^4=q_*p^*3[x]= {1\over 6} H_F^3$,
    \item ${\mathrm c}_1(\omega_p)=q^*H_F+l$,
    \item $q^*\mathrm{c}_2(\cU_F^{\vee})=l^2-lq^*H_F$,
    \item $p_*q^*H_F^2=21H_X$, $p_*(l^2)=0$, $p_*(lq^*H_F)=6H_X$,
    \item If $X$ is general, $[W]=75H_X$,
    \item $p_*({\mathrm c}_1(\omega_p)q^*[S])= 180H_X^2$,
    \item $\lambda =9H_X $,
    \item $\mathrm{c}_1(N_{S/F})=3H_S\in \HH^2(S,\QQ)$, $\mathrm{c}_2(N_{S/F})=1125$.
    \end{enumerate}
\end{lemma}
\begin{proof}
In the following we will repeatedly use the fact that if $x$ a cycle in $X$ and $y$ a cycle in $F$ with ${\mathrm codim}(x)
+ {\mathrm codim}(y)=5$, then $q_*p^*(x) \, y = x\, p_*q^*y \in {\mathbb Z}$, from the projection formula.

(1) and (3) are in \cite[Lemma 4]{amerik} (see also \cite{osy}), whereas (2), (4) are standard intersection theory of
Schubert classes combined with formula of the class of F given in \eqref{eq:class of F}.

(5) follows as we have $\omega_p=\omega_I-p^*\omega_X$, so by combining the facts that $\omega_X=-3H_X$, $p^*H_X=l$ and
$\omega_I=\omega_q=q^*H_F-2l$ we get the result. (6) is the Grothendieck relation.

For (7), the first follows as if we let $p_*q^*H_F^2=aH_X$, then $p_*q^*H_F^2 H_X^3=3a$ so, by Theorem \ref{thm:amerik and osy} and (4), we have 
\begin{align*}
p_*q^*H_F^2 H_X^3 &= H_F^2 q_*p^*(H_X^3)\\&= H^2_F \sigma_2\\& = 63.
\end{align*}
For the second, $l^2=p^*H_X^2$ implies $p_*l^2=0$. For the third, 
\begin{align*}
p_*(lq^*H_F)&= p_*(p^*H_Xq^*H_F)= H_X p_*q^*H_F\\&=6H_X,
\end{align*}
where for the final equality, if $ p_*q^*H_F = m
[1_X]$ with $m=p_*q^*H_F [x]$ for $x\in X$, then $m=H_F q_*p^*[x]= H_F [C_x]=6$.

For (8) let $[W]=aH_X$. Then $[W]H_X^3=aH_X^4=3a$. We will show
later in Proposition \ref{cor:deg qs} that for $X$ general, $p:q^{-1}(S)\to W$ is birational. Assuming this for the
moment, we have $[W]=p_*q^*[S]$. Furthermore, we have for any smooth cubic that
\begin{align*}
p_*q^*[S] H_X^3&= 
[S]q_*p^*H_X^3\\&= 5\sigma_2^2|_F \\&= 225,
\end{align*}
so that $p_*q^*[S]=75H_X$. This gives (8) if $X$ is general. 

For (9), recall from (5) that $c_1(\omega_p)=p^*H_X+q^*H_F$ and as we saw in (8) above we have the relation
$p_*q^*[S]=75H_X^2$ for all $X$. It follows that
\begin{align*}
p_*(p^*H_Xq^*[S])&= H_X p_*q^*[S]\\&= 75 H_X^2.
\end{align*}
Assuming that $X$ is very general, we have that $\HH^{2,2}(X,\ZZ)_{\operatorname{pr}}=0$ from Deligne's Invariant Cycle
Theorem (see \cite[Corollary 1.2.12]{huybrechts}). Hence $p_*(q^*H_Fq^*[S])=mH_X^2$ for some integer $m$, giving
\begin{align*}
3m=p_*(q^*H_F\, q^*[S]) H_X^2 &= p_*q^*(H_F[S])H_X^2 \\&= 
H_F [S] q_*p^*H_X^2\\&= 5H_F^2\sigma_2 = 5 \cdot 63.
\end{align*}
Therefore $p_*(q^*H_Fq^*[S])=105H_X^2$ giving the result if $X$ is very general.

To extend (9) to any cubic, note that all classes in question are defined globally on the universal family of
smooth cubic fourfolds $f:\cX\to |\OO_{\PP^5}(3)|_{\text{sm}}$ (see Section \ref{sec:surface V}). Hence, if we consider
the class $\alpha=p_*({\mathrm c}_1(\omega_p)q^*[\cS])- 180H_\cX^2\in \HH^4(\cX,\ZZ)$, then as $R^4f_*\ZZ$ is a local system
and $\alpha_t=0$ for $t$ in the complement of a countable union of closed subsets, we obtain that $\alpha_t=0$ for all
$t\in|\OO_{\PP^5}(3)|_{\text{sm}}$. The analogous result for Chow groups is \cite[Lemma 3.2]{voisinbook}. 

For (10), from the standard formula for $\lambda$ (e.g., from \cite[3.110]{harrismorrison}) and (5)  we have 
\begin{align*}
\lambda &= {1\over 12}(p_*\omega_p^2+[W]) \\ &= {1\over 12}(p_*(l+q^*H_F)^2+75H_X)\\
&=9H_X.\qedhere
\end{align*}
For (11), note that $K_F=0$ and Theorem \ref{thm:amerik and osy} imply the first equality, whereas the second
also follows from the description of $S$ as a degeneracy locus and some computations involving the Harris--Tu formula, see
\cite[Proposition 4.1]{cubicfourfolds2}.
\end{proof}

\section{The Curve of Lines Through a Point}\label{sec:curve of lines} 

In this section we collect and extend some facts about the $(2,3)$-intersection $C_x$ parametrising lines through a
point $x\in X\subset\PP^5$ of a smooth cubic fourfold. As mentioned in the introduction, if $X$ is general then $C_x$ is
always a curve, whereas for arbitrary $X$ there can be only finitely many points where it is a surface. The following is
well known to experts but we include the argument so as to fix notation and extract a bit further out of its proof.

\begin{lemma}\label{lem:sings Cx}
Let $X\subset \PP^5$ be a smooth cubic fourfold and $x\in X$ such that $C_x=p^{-1}(x)$ is 1-dimensional.
\begin{enumerate}
    \item If $\ell$ is a line of first type through $x$, then $C_x$ is smooth at $[\ell]$.
    \item If $\ell$ is a line of second type through $x$, then $C_x$ is singular at the point $[\ell]$.
\end{enumerate}
\end{lemma}
\begin{proof}
Let $\ell$ be a line of first type in $X$. We may assume that it is given by the equations
$\VV(x_2,x_3,x_4,x_5)$. Then the equation of $X$ may take the form \cite[6.9]{cg}.

\begin{align}\label{eqtype1}
\begin{split}
F=&\ x_4x_0^2+x_5x_0x_1+x_3x_1^2+ x_0Q_0(x_2,x_3,x_4,x_5)+ \\& x_1Q_1(x_2,x_3,x_4,x_5)+
 P(x_2,x_3,x_4,x_5)=0.
\end{split}
\end{align}

Since $x\in \ell$ we may assume that $x=[1,a,0,0,0,0] \in \ell$ (the case $x=[0,1,0,0,0,0]$ is treated similarly). To
find $C_x$, we write the equation as
\begin{align*}
&x_4x_0^2+x_5x_0(x_1-ax_0) + ax_5x_0^2+x_3(x_1-ax_0)^2
+2ax_0x_3(x_1-ax_0)\\&+a^2x_0^2x_3
+ x_0Q_0 + (x_1-ax_0)Q_1+ax_0Q_1+ P=0.
\end{align*}
Putting $x_0=1$ and $x'_1=x_1-a$ we get
\[
 [a^2x_3+x_4+ax_5]+[x_5x'_1+2ax_3x'_1+Q_0+aQ_1]+[x_3{x'_1}^2+x'_1Q_1+P]=0.
\]
Note in this case that $\Pi_{\ell}= \VV(x_3,x_4,x_5)$ is the tangent 2-plane to X along $\ell$. 
The lines through $x=(0,0,0,0,0)\in {\mathbb A}^5$ (in the $x_1',x_2,x_3,x_4,x_5$ coordinates) are determined by
their slopes $[x'_1:x_2:x_3:x_4:x_5]\in \PP^4$, with the slope of the line $\ell $ being $[1:0:0:0:0]$. They are
parametrised by the following curve which is given by the system
\begin{align*}
&a^2x_3+x_4+ax_5=0,\\
&T_2= x_5x'_1+2ax_3x'_1+Q_0(x_2,x_3,x_4,x_5)+aQ_1(x_2,x_3,x_4,x_5)=0,\\
&T_3= x_3{x'_1}^2+x'_1Q_1(x_2,x_3,x_4,x_5)+P(x_2,x_3,x_4,x_5)=0.
\end{align*}
Substituting $x_4$ from the 1st equation, $T_2=0, T_3=0$ become equations in
the variables $x'_1, x_2,x_3,x_5$ in $\PP^3$. The point in $C_x$ which corresponds to the line $\ell$ is then $[1:0:0:0]$ (in
the $x'_1,x_2,x_3,x_5$-coordinates). At this point the gradients of the surfaces $T_2,T_3$ are $\langle0,0,2a,1\rangle$ and
$\langle0,0,1,0\rangle$ respectively. Hence the intersection is transversal and the curve $C_x$ is smooth at the
point $[\ell]$ with the tangent line given by $x_3=x_5=0$ in $\PP^3$. Note then that the plane spanned by the line $\ell $ and the above tangent line is the tangent plane $\Pi_{\ell}$  to $\ell $ in $X$. 

Let now $\ell$ be a line of second type, given again by $\VV(x_2,x_3,x_4,x_5)$. Then the equation of $X$ may take the form
\cite[6.10]{cg}
 \begin{align} \label{eqtype2}
 \begin{split}
F=&\ x_4x_0^2+x_5x_1^2+ x_0Q_0(x_2,x_3,x_4,x_5)+ x_1Q_1(x_2,x_3,x_4,x_5)
 \\&+P(x_2,x_3,x_4,x_5)=0.
 \end{split}
 \end{align}

A similar calculation shows that with centre the point $x=[1:a:0:0:0:0] \in \ell$ (the case $x=[0:1:0:0:0:0]$ is
treated similarly) the equation of $X$ becomes (by putting $x_0=1$ and $x_1-a=x'_1$)
\[
[x_4+a^2x_5] + [2ax_5x'_1+Q_0+aQ_1] +[x_5{x'_1}^2 +
x'_1Q_1 +P] =0.
\]
Note in this case that $H_{\ell}= \VV(x_4,x_5)$ is the tangent 3-space to $X$ along $ \ell$.  Then $C_x$ is the
intersection of two surfaces in a 3-dimensional projective space in variables $x_1', x_2,x_3,x_5$  with gradients at    
$\langle0,0,0,2a\rangle$ and $\langle0,0,0,1\rangle$ respectively at the point $[1:0:0:0]$ corresponding to $\ell$. Hence the intersection is not transversal at the point $[\ell]$, $C_x$
is singular there with the tangent cone given by the plane  $x_5=0$ in $\PP^3$. Note then that the space
spanned by the line $\ell $ and the above tangent cone is the tangent space $H_{\ell}$  to $\ell $ in $X$. 
\end{proof}

To summarise the geometry, for $x\in \ell\subset X$, the tangent space to the line $\ell\subset X$ is spanned by $\ell$
and the tangent cone to $C_x$ at the point $[\ell]$. This space has dimension 2 (i.e., $[\ell]$ is a line of first type)
if and only if $C_x$ is smooth at $[\ell]$ and has dimension 3 (i.e., $[\ell]$ is a line of second type) if and only if
$C_x$ has a singularity at $[\ell]$ (with tangent cone of dimension two). When we consider $C_x$ as a curve in $F$
(i.e., as $q(p^{-1}(x))$), the singular points correspond to the intersection points with the surface $S$. When
$x\notin W$ this intersection is empty and the curve $C_x$ is smooth. 

\begin{remark}\label{rem:tangent hyperplane}
In \cite[Definition 6.6]{cg}, lines of first or second type are defined as those for which the image of the dual mapping
is a smooth plane conic or a two-to-one covering of the projective line respectively. 

In fact, when the above line $\ell$ is of second type, $\nabla F([s:t:0:0:0:0])=\langle0,0,0,0,s^2,t^2\rangle$ and the
fibres of the two-to-one covering are formed by the points $[s:\pm t:0:0:0:0]$, with ramification points the two points
at infinity $x_0=[1:0:0:0:0:0], \ x_1=[0:1:0:0:0:0]$.  With $x=[1:a:0:0:0:0]\in \ell$, $a\neq 0$ and $Y_x= T_xX \cap X$,
one sees that $Y_x$ contains $\ell$ and has two singular points on $\ell$, namely $x$ and $x'=[1:-a:0:0:0:0]$, the
conjugate of $x$ under the dual mapping. For $a=0$, $Y_{x_0}$ has a non-ordinary singularity at $x_0$ and the same for
$Y_{x_1}$. Finally note that for a point $x\in \ell$, with $[\ell]\in S$, the singularity of $C_x$ corresponding to $\ell$ has the same
type as the section $Y_x=Y_{x'}$ has at the conjugate point  $x'$ (see \cite[Section 2, pg.6, case (i)]{w},
\cite[Section 3.1]{cml}). Hence if $x'$ is not in the Hessian this singularity is nodal. If there are several second
type lines containing $x$, each one induces a singularity as above. 

On the other hand, if $[\ell]\in F\setminus S$ and $x\in \ell$
then $Y_x$ does not have a singularity on $\ell$ other than $x$. In particular, if $x\notin W$ then the only
singularity of $Y_x$ is at $x$.
\end{remark}

\begin{lemma}\label{lem:g4corr}
    We have the following two correspondences 
    \begin{enumerate} 
        \item between pairs $(Y,x)$ where $Y\subset\PP^4$ is a cubic threefold with a singular ordinary double point
        $x$ (resp.\ $A_2$ singularity at the point $x$), and smooth non-hyperelliptic curves of genus 4 with
        canonical image which lies on a smooth quadric (resp.\ quadric cone).
        \item between pairs $(Y,x)$ where $Y\subset\PP^4$ is a cubic threefold with two
        ordinary double points so that $x$ is one of them, and smooth non-hyperelliptic genus 3 curves with two marked
        points and with nodal canonical image lying on a smooth quadric.
    \end{enumerate}
     In both cases the canonical image of the corresponding curve parametrises the lines in $Y$ through $x$.
\end{lemma}
\begin{remark}
The cases of (1) above are exactly the smooth $(2,3)$ complete intersections in $\PP^3$. Note that in the case of a
quadric cone, as the curve is smooth, it necessarily does not pass through the vertex of the cone. In the second, the
line joining the two singularities of $Y$ is necessarily contained in $Y$ by Bezout.
\end{remark}
\begin{proof}
    The first has appeared many times in the literature, see, e.g., \cite[p.306-307]{cg} and \cite[\S 3.1]{cml}. The
    second is a minor modification of this so we briefly sketch the construction. For
    $(C',p_1,p_2)\in\mathcal{M}_{3,2}^{\mathrm nhyp}$, the linear system $|K_{C'}+p_1+p_2|$ gives a morphism $C'\to\PP^3$
    with image a sextic curve $C$ with one node $p$. Just as in the classical construction of the canonical embedding of
    a non-hyperelliptic genus 4 curves, $\HH^0(\PP^3,  I_C(2))=\CC$ so there is a unique quadric $Q$ containing $C$,
    which we may assume is smooth by genericity. Also, the linear system $|I_C(3)|$ induces a birational map $h:\PP^3
    \dashrightarrow \PP^4$ with image a singular cubic threefold $Y$. This can also be realised by blowing up $\PP^3$ at
    the curve $C$ to obtain a variety $\tilde{Y}$ with one node ($C$ and $\tilde{Y}$ have the same singularity count and
    type) lying above the node of $C$, and blowing down the strict transform of the quadric $Q$ to obtain $Y$. The
    threefold $Y$ will be singular both at the image $x$ of the quadric but also the image of the singular point of
    $\tilde{Y}$. Projecting from $x$ gives the inverse map to $h$.
\end{proof}

Note that from Lemma \ref{lem:wirreddiv}, if $X$ is general, the locus $W\subset X$ spanned by lines of second type is an
irreducible divisor.

\begin{proposition}\label{prop:properties C_x}
    Let $X\subset\PP^5$ be a general cubic fourfold and $W\subset X$ the locus spanned by lines of second type. For any
    $x\in X$ the curve $C_x\subset F$ parametrising lines passing through $x$ is reduced and connected. Moreover, the following
    hold.
    \begin{enumerate}
    \item For $x\in X\setminus W$ the curve $C_x$ is irreducible and smooth of genus 4 and if $x$ is general then it is
    general in moduli.
    \item For $x\in W$ general, $C_x$ is irreducible and has only one node and its normalisation is a
    general curve of genus 3 with the two preimages of the node being general points.
    \item For any $x\in W$ the total Tjurina number of $C_x$ is $\leq6$, and if reducible, $C_x$ is the union of two
    irreducible, at worst nodal, plane cubics meeting at 3 distinct points. Hence, there can be at most 4 lines of
    second type through $x$.
    \end{enumerate}
\end{proposition}
\begin{proof}
    From \cite[Lemma 2.7]{lsv} we know that $C_x$ is reduced for all $x\in X$, and that it is connected follows from
    \cite{barthvdv}. For any $x$ not in $W$, from Lemma \ref{lem:sings Cx} it follows that $C_x$ is also smooth and
    hence irreducible, proving the first part of (1). We now prove the remainder of (1) and (2).
    
    For (1), let $C$ be a curve of genus 4 which is general in moduli (resp., for (2) a general arithmetic genus 4 curve
    with one node). From Lemma \ref{lem:g4corr} we obtain a cubic threefold $Y$ with one (resp.\ two) nodes, in such a
    way that $C$ parametrises lines contained in $Y$ passing through the marked node $p_0$. Write the equation
    determining $Y$ as $F(x_0,\ldots,x_4)=0$ with singularity at the marked node $p_0=[1:0:\ldots:0]$ (in such a way
    that the other node is not on the $x_0=0$ hyperplane), and consider the cubic equation
    \[
    G(x_0,\ldots, x_5)=F(x_0,...,x_4)+x_0^2x_5.
    \]
    It is easy to check that $X_0=\VV(G)$ defines a smooth cubic fourfold with $Y$ as the tangent hyperplane section at the
    point $[1:0:\ldots:0]$, so that we have $p_0\in X_0$ and that $C=C_{p_0}$ is the curve of lines through $p_0$. 
    Note that in the second case, $Y$ will be a hyperplane section of $X_0$ with two nodes, and the line between these
    nodes is of second type (cf.,\ Remark \ref{rem:tangent hyperplane} for the converse statement that the tangent
    hyperplane at a general point of a general line of second type has two singularities). 

    Let now \[\mathcal{X}\to |\OO_{\PP^5}(3)|_{\mathrm{sm}}\] the universal family of smooth cubic fourfolds,
    $p:\mathcal{I}\to\mathcal{X}$ the projection from the universal family of lines on all of them, and
    $q:\mathcal{I}\to \cF$ the morphism to the universal Fano variety of lines (we discuss this and some of the
    universal families mentioned below in the beginning of the following section). From semicontinuity of fibre
    dimension, the generic fibre of the morphism $p:\mathcal{I}\to\mathcal{X}$ is a smooth curve of genus $4$. In fact,
    \cite[Lemma 2.1, Corollary 2.2]{coskunstarr} prove that for any smooth cubic fourfold $X$, there are only finitely
    many points in $X$ over which the fibre of $p:I\to X$ is not 1-dimensional, so this locus is codimension $4$ in
    $\mathcal{X}$. The universal locus $\mathcal{S}\subset \cF$ of second type lines admits a projection to the
    Grassmannian $\G(2,6)$ with isomorphic, irreducible fibres (see \cite[Proposition 2.2.13]{huybrechts} and its
    proof). Hence $\mathcal{S}$ is irreducible and of codimension 2 in $\cF$ so that the total space
    $q^{-1}(\mathcal{S})$ of second type lines in $\mathcal{I}$ is irreducible and likewise its image
    $\mathcal{W}\subset \mathcal{X}$, the locus of points on lines of second type, which is a divisor in ${\mathcal X}$
    since $q^{-1}(\mathcal{S})\to\mathcal{W}$ is generically finite from Lemma \ref{lem:wirreddiv}. We conclude that the
    non-empty open subset of $\mathcal{W}$ parametrising points where the fibres of $p$ are curves is irreducible.

    Since a general $X$ degenerates to $X_0$, so will the corresponding family of curves $C_x$ degenerate to $C$,
    proving (1), since being general is an open condition. Similarly, the curve $C_x$ for a
    general point $x\in W\subset X$ must be general and irreducible with one node, since being nodal is
    an open condition in a flat family of curves and $\mathcal{W}$ is irreducible. This completes (2).

    We now prove (3). As mentioned in Section \ref{sec:background}, the class of $C_x$ in cohomology is $2\beta$, and
    since for the very general $X$ the class $\beta$ generates the algebraic part of cohomology, it must be the case
    that for general $X$ the class $\beta$ is irreducible, so that every curve representing it must be irreducible,
    implying $C_x$ has at most two irreducible components. From \cite[Lemma 3.8]{no} it follows that each such
    irreducible component has arithmetic genus $1$ and at worst nodes, implying that the two components must meet at
    three distinct points as the total arithmetic genus is 4. The statement about singularities follows from the
    correspondence between singularities of the blow up $\tilde{Y}_x$ at $x$ of the tangent hyperplane section
    $Y_x=T_xX\cap X$ and those of $C_x$ of \cite[3.1-3.3]{cml} and from the fact that for a general $X$, every
    hyperplane has isolated singularities of Tjurina number at most $6$ from \cite[Corollary 3.7]{lsv}. As the
    arithmetic genus of $C_x$ is four and from Lemma \ref{lem:sings Cx} every singularity corresponds to a second type
    line through $x$, the above discussion implies that there can be at most 4 lines of second type through $x$. 
\end{proof}

\begin{remark}\label{rem:sings}
    In particular, from Lemma \ref{lem:sings Cx} the curve $C_x$ can have at most four singular points for $x$ an
    arbitrary point of a general cubic fourfold. It would be desirable to have a complete list of possible singularities
    which occur for $C_x$ and a description of the corresponding loci in $X$.
\end{remark}

\begin{corollary}\label{cor:deg qs}
    For $X\subset\PP^5$ a general cubic fourfold, the morphism $p:\PP(\cU_S)\to W$ is birational and the degree of $W$
    in $\PP^5$ is 225.
\end{corollary}
\begin{proof}
    That through a general point $x\in W$ there passes precisely one second type line follows from the fact that each
    second type line through $x$ forces an extra singularity of $C_x$ by Lemma \ref{lem:sings Cx}, so we conclude from
    Proposition \ref{prop:properties C_x}, as the generic $x\in W$ has $C_x$ irreducible with one node.
    That the degree of $W$ is 225 was already computed in Lemma \ref{lem:intersection theory}, under the assumption that
    the degree of the above map is one.
\end{proof}

\begin{remark}\label{rem:higher dim}
    More generally, let $n\geq3$ and let $X\subset\PP^{n+1}$ be a general cubic $n$-fold, $F$ its Fano scheme of lines
    and $F_2\subset F$ the locus of lines of second type, which is smooth of dimension $n-2=\frac12\dim F$ from
    \cite[Proposition 2.2.13]{huybrechts}. Consider the induced morphism from the family of lines 
    \[\xymatrix{
        \PP(\cU_{F_2})\ar[r]^-p\ar[d]^q & W\subset X \\
        F_2 &
    }\]
    for $\cU$ the universal bundle on the Grassmannian $\G(2,n+2)$. The same argument as in Proposition
    \ref{prop:properties C_x} proves that $p$ is birational onto its image, a divisor in $X$, and we summarise this
    in the following two paragraphs.
    
    First, there is a correspondence between pairs $(Y,x)$ where $Y\subset\PP^n$ is a cubic $(n-1)$-fold with two
    ordinary double points, one of which is $x$, and $(2,3)$-complete intersections in $\PP^{n-1}$ with one
    ordinary double point. For example (see \cite{dgk}), if $n=3$, this will be between cubic surfaces with two 
    nodes, and $0$-cycles in $\PP^2$ of length 6 with support on 5 distinct points, one of which is non-reduced. To
    explain the correspondence, call $Q$ and $C$ the quadric and cubic respectively, and $M$ their intersection. The
    correspondence follows from the diagram 
    \[\xymatrix{
        & \operatorname{Bl}_x(Y)\ar[dl]_{\pi_1}\ar[dr]^{\pi_2} &\\
        Y\ar@{-->}[rr]^{h} & & \PP^{n-1} 
    }\]
    where $\pi_1$ is the blowup map, $\pi_2$ contracts the family of lines through $x$ and $h$ is the projection from
    the marked node. The image of the exceptional divisor $\pi_1^{-1}(x)$ under $\pi_2$ is the quadric $Q$, whereas
    $\pi_2$ is the blowup of the intersection $M$ of $Q$ with a cubic $C$. The sheaf $I_M(3)$ has $n+1$ sections, and this
    linear system induces the inverse rational map to $h$ with image $Y$. Starting with $M$, the
    blowup $\operatorname{Bl}_M(\PP^{n-1})$ is singular at a point $y$, and the strict transform of $Q$ can be contracted
    to a point $x$, in a cubic $(n-1)$-fold $Y$ with ordinary double points at $x,y$.

    Now for $x\in X$ a general point on a general line $\ell$ of second type, the tangent hyperplane section $Y_x=T_xX\cap X$ will have a
    node at $x$ and at precisely one other point on $\ell$ from Remark \ref{rem:tangent hyperplane}. The corresponding
    $(2,3)$-complete intersection parametrises lines through $x$ and will have precisely one ordinary double point. As
    in the fourfold case, this implies the map $p$ is birational, or that the ramification of the map $p:I\to X$ is
    generically as simple as possible along $W$.
\end{remark}

To summarise the discussion above, we have obtained the following.
\begin{corollary}\label{cor:general}
If $x\in X\subset\PP^5$ is a general point of a general cubic fourfold, then $C_x$ is a general genus 4 curve and the morphism
\[p|_{q^{-1}(S)}:q^{-1}(S)\to W\subset X\] is birational onto its image $W$, an irreducible divisor in $X$. If $x\in W$
is general, then $C_x$ is an irreducible curve with one node whose normalisation is a general genus 3 curve.
\end{corollary}

We give now a geometric interpretation of the fact that $\ell$ and its residual $\ell'$ (i.e., the image under
the Voisin map) meet at a point $x\in\ell\cap\ell'\subset X$ in terms of ramification points of the $g^1_3$'s on $C_x$. In
particular, we will be interested in the degenerate case $\ell=\ell'$.
 
\begin{lemma}\label{lem:cg}
Let $(Y,x)$ be a pair of a cubic threefold $Y\subset\PP^4$ and a point $x\in Y$ as in either of cases (1) or (2) of Lemma
\ref{lem:g4corr} and let $C_x$ be the sextic curve parametrising lines through $x$. For $t\in C_x$, we denote the
corresponding line by $\ell_t$. Then the following are true.
\begin{enumerate}
    \item For $t_1,t_2\in C_x$ and $\langle\ell_{t_1},\ell_{t_2}\rangle\cap Y=\ell_{t_1}+\ell_{t_2}+\ell$, we have that
    $[\ell]\in C_x$ if and only if $t_1, t_2$ are in the same ruling of the quadric containing $C_x$. In this case
    $\ell=\ell_{t_{12}}$ for $t_{12}\in C_x$ the third point of intersection of $C_x$ with that ruling.
    \item The residual line $\ell_{t'}$ in the intersection of $Y$ with a 2-plane tangent to the cone of lines through
    $x$ at the line $\ell_t$ also passes through $x$ if and only if $t$ is a ramification point of one of the possibly
    two $g^1_3$'s that $C_x$ possesses. In particular, $\ell_t=\ell'_t$ if $t$ is a triple ramification point of one of
    the possibly two $g^1_3$'s that $C_x$ possesses.
\end{enumerate}
\end{lemma}
\begin{proof}
    The first part is basically \cite[Lemma 8.6]{cg}, whereas the second follows from the first.
\end{proof}

Recall that for a cubic fourfold $X\subset\PP^5$, the divisor $W\subset X$ is the locus of points $x$ whose curve $C_x$
is singular. Proposition \ref{prop:properties C_x} implies that if $W'\subset W$ is the closure of the set 
\begin{align}\label{eq:W'}
\{x\in X : C_x \text{ has worse than irreducible 1-nodal singularities}\}
\end{align}
then $W'$ has codimension at least 2 in $X$.

\begin{proposition}\label{prop:ram pts}
    Let $X\subset\PP^5$ be a general cubic fourfold and $x\in X\backslash W'$. For $\xi$ a $0$-dimensional subscheme of
    length 2 on $C_x$, i.e., $\xi=t_1+t_2$ for two possibly non-distinct points in $C_x$, the 2-plane spanned by
    $\ell_{t_1},\ell_{t_2}$ (and in the case $t_1=t_2$, a 2-plane tangent to $\ell_{t_1}$) meets $X$ at a line $\ell_{t_3}$ which passes through $x$ if and only if $t_1,t_2,t_3$ are a
    fibre of one of the possibly two $g^1_3$'s on $C_x$.
\end{proposition}
\begin{proof}
    Apply the previous lemma to the nodal cubic threefold $Y=T_xX\cap X$, noting that such lines $\ell_{t_i}$ are contained
    in $Y$  - recall, also, the  summary of the geometry following Lemma \ref{lem:sings Cx}.
\end{proof}

Suppose we are given a pair $(\ell,\Pi)$, with $\ell $ a line in $X$ and $\Pi$ a  tangent 2-plane to $X$ along $\ell$, and a
point $x\in \ell\backslash W'$. Let $\ell=\ell_t$ and let $\ell'=\ell_{t'}$ the residual line  in the intersection of
$X$ with $\Pi$, with $t,t' \in C_x$. Then by the above proposition we have that $ x\in \ell_t \cap \ell_{t'}$  if and
only if $t$ is  a ramification point of one of the (possibly) two $g^1_3$'s on $C_x$ and $t'$ the residual to that
ramification point in the corresponding $g^1_3$. In particular we have $\ell'=\ell=\ell_t$ if and only if $t$ is a
triple ramification point. In Section \ref{sec:ramloci} we will study the locus in the universal family of lines $I$ of
ramification points (and their residuals) of the $g^1_3$'s of the curves $C_x$.

From Proposition \ref{prop:properties C_x} there is a rational map 
\[j: X \dashrightarrow\overline{\mathcal M}_4\]
not defined on a codimension two locus, and whose image hits the boundary in the component $\Delta_0$ of 1-nodal curves.
In the following, denote by $D_3\subset\bar{\cM}_4$ the divisor that is the closure of the divisor parametrising smooth
genus four curves that have a $g^1_3$ that admits a triple ramification point.

\begin{lemma}\label{lem:D3}
    Let $X$ be a general cubic fourfold and consider $j:X\dashrightarrow\bar{\cM}_4$ the induced moduli map, taking a
    point $x$ to $[C_x]$. Then \[j^*[D_3] = 126 H_X.\]
\end{lemma}
\begin{proof}
    From \cite{harris2}, the class of $D_3$ on the moduli space can be computed as follow \[ [D_3]=2(132\lambda
    -15\delta_0+\mbox{other boundary components}). \] By Lemma \ref{lem:intersection theory} we get that \[j^*[D_3] =
    2(132\cdot 9 - 15\cdot 75)H_X= 126 H_X.\qedhere\]
\end{proof}

\begin{corollary}\label{cor:Vnonempty}
    Let $X\subset\PP^5$ be a general cubic fourfold. The locus $V\subset F$ of triple lines is non-empty.
\end{corollary}
\begin{proof}
    From Lemma \ref{lem:D3} we obtain that the locus $x\in X$ so that $C_x$ has a triple ramification point is non-empty
    and divisorial. In particular, from Lemma \ref{lem:cg} and Proposition \ref{prop:ram pts} we see that $V$ must be
    non-empty as required.
\end{proof}

\begin{remark}
    It is also possible to give an equivalent to Corollary \ref{cor:deg qs} for the locus $V$ of triple lines. We will
    later in Theorem \ref{thm:class of V} compute that $[V]=21\mathrm{c}_2(\cU_F)\in {\rm CH}^2(F)$, so as from Lemma
    \ref{lem:intersection theory}, $q_*p^*H_X^3=\sigma_2|_F$, we get $21\mathrm{c}_2(\cU_F)\sigma_2 =
    21(\mathrm{c}_2(\cU_F)H_F^2 - \mathrm{c}_2^2(\cU_F))=3\cdot 126$ and as $H_X^4=3$, we conclude that 
    \[p_*q^*[V] = 126H_X.\]
    As this agrees with $j^*D_3$ above, we obtain that the morphism
    \[q^{-1}(V) \to p(q^{-1}(V))\]
    is birational. In particular, the generic point $x\in X$ that has a triple line through it has exactly one triple
    line through it. 
\end{remark}

\section{Universal Families and the Surface $V$}\label{sec:surface V}

Throughout this section $X\subset\PP^5$ will be a general cubic fourfold, $F$ its Fano variety of lines and we will denote
by $V,S$ etc.\ as in Section \ref{sec:background}.
We will often consider the following special linear subspaces in $\PP^5$ 
\begin{equation}\label{eq:triple0}
    \ell_0=\VV(x_2,x_3,x_4, x_5), \  \Pi_0=\VV(x_3,x_4, x_5),\ H_0=\VV(x_4,x_5).
\end{equation}

Recall first that the Voisin map $\phi:F\dashrightarrow F$ is resolved by blowing up the locus $S$ to give Diagram
\eqref{diag:voisindiagram}. 
Consider now the incidence variety 
\begin{align*}
\tilde{F}=\{(\ell, \Pi) : \Pi\text{ is tangent to }X\text{ along }\ell\}\subset\G(2,6)\times\G(3,6).
\end{align*}
\begin{lemma}(\cite[Remark 2.2.18]{huybrechts})\label{lem:huybrechtsincidence} 
    The blowup $\Bl_S(F)$ is isomorphic to the incidence variety $\tilde{F}$, so in particular the fibre
    $E_{S,[\ell]}\cong\PP^1$ of $\pi_{E_S}:E_S\to S$ over $[\ell]$ parametrises the 2-planes tangent along $\ell$. 
\end{lemma}

The map $\psi$ of \eqref{diag:defpsi} lifts to a rational map $\psi : \tilde{F} \dashrightarrow X$, not defined at the
locus of pairs $(\ell, \Pi)$ as above for which $\Pi \cap X = 3\ell$. This is the strict transform $\tilde{V}$ in
$\tilde{F}$ of the locus $V$ of Definition \ref{triplelinelocus}. We will see later in Corollary
\ref{cor:resolvepsitilde} that the map $\psi$ is resolved by blowing up $\tilde{V}$.

In this section we will need to construct the universal families of the varieties mentioned above and derive various of
their properties from these constructions. First, we note that we may construct the universal cubic hypersurface
$\cX\subset\PP^5\times|\OO_{\PP^5}(3)|$ as the vanishing locus of the universal section
$s_\cX\in\HH^0(\PP^5\times|\OO_{\PP^5}(3)|,\OO(3,1))$, and it is not hard to also construct the universal family
$\cF\subset |\OO_{\PP^5}(3)|\times\G(2,6)$, whose fibre under the first projection over a cubic $X=\VV(G)$ is the Fano
scheme of lines $F$ of $X$ (see \cite[\S 2]{huybrechts}). The second projection realises $\cF$ as a projective bundle
$\PP(K)$ where $K$ is a rank 52 vector bundle on $\G(2,6)$ sitting in the following sequence
\begin{equation}\label{eq:defK}
0\longrightarrow K\longrightarrow S^3(\CC^6)^\vee\otimes\OO_{\G(2,6)}\longrightarrow S^3\cU^\vee\longrightarrow 0
\end{equation}
whose fibre over a line $[\ell]\in\G(3,6)$ is $\PP(\HH^0(\PP^5,  I_\ell\otimes\OO(3)))$. Finally, it was proven in
\cite[Lemma 1]{amerik} (see also \cite[\S 2]{huybrechts}) that the universal family $\cS\subset\cF$ of surfaces of
second type lines, with fibre over a smooth $X={\VV}(G)$ the surface $S$ of type two lines contained in $X$, is smooth
when restricted to the locus of smooth cubics, and has non-empty fibres there.

We continue now by exploring universal constructions for $\tilde{F}$ and $\tilde{V}$. Define first the two 
spaces
\begin{align}\label{eq:cG}
\begin{split}
\cG&=\left\{(\ell, \ell', \Pi): \ell,\ell' \subset \Pi \right\} \subset \G(2,6) \times \G(2,6) \times \G(3,6),\\
\cG_\Delta&=\left\{(\ell,\ell',\Pi)\in\cG : \ell=\ell'\right\}.
\end{split}
\end{align}
Note that $\cG$ is smooth, irreducible and of dimension 13 since it is the self-product of the universal family of
lines contained in a 2-plane in $\PP^5$, i.e.,\ is a $\bP^2 \times \bP^2$-fibration over the 9-dimensional $\G(3,6)$.
Similarly, $\cG_\Delta$ is smooth, irreducible and of dimension 11 and in this case we refer to its points simply as
tuples $(\ell, \Pi)$.  As it will be needed in various statements in the rest of this section, we detail the local
structure of $\cG$.

Our aim is to describe an affine open set $U$ of ${\mathcal G}$. A choice of local coordinates $x=(x_{03},...,x_{25})$
of an affine open $\CC^9$ in $\G(3,6)$ determines a 2-plane $\Pi(x) $ given
by the row space of the matrix
\[ 
A(x)= 
\begin{pmatrix}
1& 0& 0 & x_{03} & x_{04} & x_{05} \\
0& 1& 0 & x_{13} & x_{14} & x_{15} \\
0& 0& 1 & x_{23} & x_{24} & x_{25}
\end{pmatrix}. 
\]
Then, for $y=(y_1,y_2)$ affine coordinates of $\CC^2$, a line $\ell(x,y)\subset\Pi(x)$ is the row space of the matrix
\[
B(x,y)=
\begin{pmatrix}
1& 0& y_1 & x_{03}+x_{23}y_1 & x_{04}+x_{24}y_1 & x_{05}+x_{25}y_1 \\
0& 1& y_2 & x_{13}+x_{23}y_2 & x_{14}+x_{24}y_2 & x_{15}+x_{25}y_2
\end{pmatrix}.
\]
In parametric form, if $[s_0:s_1:s_2]$ are the coordinates of $\PP^2$, then the 2-plane $\Pi(x)$ in $\PP^5$ is given by
$[(s_0,s_1,s_2)A(x)]$, whereas for $[\lambda:\mu]$ coordinates of $\PP^1$, the line $\ell(x,y)$ is given by
$[(\lambda,\mu)B(x)]$, and its Cartesian equation in $\bP^2\cong \Pi(x)$ is $s_0y_1+s_1y_2-s_2=0$. We then define 
\begin{align}\label{eq:U}
U =\left\{\left(\ell(x,y), \ell(x,y'),\Pi(x)\right) : (x,y,y')\in\CC^9\times\CC^2 \times \CC^2 \right\}.
\end{align}
Geometrically, the points of $U$ correspond to projective planes $\Pi(x)$ in $\PP^5$ which do not intersect the projective plane
$\Pi(0)=\Pi_0=\VV(x_0,x_1,x_2)$ and lines $\ell, \ell'$ in $\Pi(x)$ which do not contain the point $[0,0,1]$ of $\bP^2 \cong \Pi
(x)$. Note that when $y=y'$ then we are in the subspace $\cG_\Delta$, the intersection is non-empty and for
$x=y=y'=0$ we are centered at the point $(\ell_0,\Pi_0)\in \cG_\Delta$. 

For $X=\VV(G)$ a cubic fourfold, denote by
\begin{equation} \label{eq:pr1tF}
\begin{aligned}
    \tilde{\cF}= &\{(X,(\ell,\ell',\Pi)) : X \cap \Pi=2\ell+\ell'\} \subset 
    |\OO_{\PP^5}(3)| \times {\mathcal G}, \\
      \tilde{\cV}= &\{(X,(\ell,\ell',\Pi)) : \ell=\ell' \} \subset \tilde{\cF}.
\end{aligned}
\end{equation}
We have a diagram
\[
\xymatrix{
    \tilde{\cF} \ar[r]^-{\mathrm{pr}_1}\ar[d]^{\mathrm{pr}_2} & 
    |\OO_{\PP^5}(3)| \\
    {\mathcal G} &
}
\]
and observe that $\tilde{\cV}=\mathrm{pr}_2^{-1}(\cG_\Delta)$. The fibre of $\tilde{\cF}$ over a general point $X\in
|\OO(3)|$ is the graph of the Voisin map, i.e., isomorphic to the blowup $\tilde{F}$ of $F$ at $S$ under the map $
(X,(\ell,\ell',\Pi)) \mapsto (X,(\ell,\Pi))$ (see Lemma \ref{lem:huybrechtsincidence}). The fibre of $\tilde{\cV}$ is
the strict transform $\tilde{V}$ of the locus $V$ of triple lines in $X$ in the blowup $\tF$, which as we saw in
Corollary \ref{cor:Vnonempty} is non-empty. 

\begin{proposition}\label{prop:tcFtcVsmooth}
    The variety $\tilde{\cF}$ is the projectivisation of a rank 47 vector bundle over $\cG$ and $\tilde{\cV}$ is its
    restriction to $\cG_\Delta$. In particular they are both smooth.
\end{proposition}
\begin{proof}
    The action of $\mathrm{PGL}(5)$ on $\bP^5$ induces a transitive action on both $\tilde{\cF}\backslash \tilde{\cV} $
    and $\tilde{\cV}$ which commutes with the projections to $\cG\backslash \cG_\Delta $ and $\cG_\Delta$ respectively
    and sends fibres isomorphically to fibres. The $\mathrm{PGL}(5)$-orbit of $U$ (from \eqref{eq:U}) in $\cG$ gives a
    covering of $\cG$ by affine open sets isomorphic to $U$ over which the fibrations $\tilde{\cF}$ and $\tilde{\cV}$
    are isomorphic to the corresponding ones over $U$. We show first that the above
    universal families are, over $U$, projectivisations of affine bundles. Taking this for granted, we have a diagram
    \begin{align}\label{eq:affinecones}
    \xymatrix{
    \tilde{\cF}^a_U\ar[r]^-{{\mathrm pr}_1} &
        M=\HH^0(\PP^5, \OO_{\PP^5}(3)) \\
    \tilde{\cV}^a_U\ar[ur]_-{{\mathrm pr}_1}\ar@{^{(}->}[u] &
    }
    \end{align}
    namely the affine bundle version of that in \eqref{eq:pr1tF}. 
     
    We now show the last claim. Let $w=(w_{k_0\cdots k_5} : \sum k_i=3)$ be coordinates for $M\cong\CC^{56}$, and for
    one such $w$, denote by $G_w$ the corresponding cubic hypersurface. From the above local analysis,
    $\tilde{\cF}^a_U\subset M\times U$ consists of those $(G_w,(\ell(x,y), \ell(x,y'),\Pi(x)))$ so that
    \begin{align}\label{eq:Gw}
    \begin{split}
    G_w(&s_0,s_1,s_2, s_0x_{03}+s_1x_{13}+s_2x_{23}, s_0x_{04}+s_1x_{14}+s_2x_{24}, \\
    & s_0x_{05}+s_1x_{15}+s_2x_{25})=c(s_0y_1+s_1y_2-s_2)^2(s_0y'_1+s_1y'_2-s_2)
    \end{split}
    \end{align}
    for some constant $c$. Letting $z$ be a coordinate for $\CC$ parametrising $c$, we have that $\tilde{\cF}^a_U$ is
    the image of the locus $\cA\subset M\times U \times\CC$ cut out by the above equation \eqref{eq:Gw} (with $c$
    replaced by $z$). This is a cubic polynomial equation in the variables $s_0,s_1,s_2$ and therefore corresponds to a
    system of 10 polynomial equations $E_{t_0t_1t_2}$, linear in $w, c$, indexed by $s_0^{t_0}s_1^{t_1}s_2^{t_2}$ so
    that $t_0+t_1+t_2=3$.  They have the form
    \[
    E_{t_0t_1t_2}:\; w_{t_0t_1t_2000} + (\mbox{terms not containing } w_{t_0t_1t_2000}) =0,
    \] 
    so that the Jacobian matrix of the system, which is a $10\times 70$ matrix, is such that the first $10\times 10$
    minor is the identity. Hence, the above system defines a smooth complete intersection of dimension $60$. Note also
    that $c=-G_w(0,0,1, x_{23},x_{24},x_{25})$, so that $\cA$ is bijective onto its image $\tilde{\cF}^a_U \subset M\times
    U$. Moreover when we fix a point $(x,y,y')\in U$, the fibre of $\cA$ over that point (resp.\ its image in $M\times U$)
    is a linear subspace of $M\times \CC \cong\CC^{57}$ (resp.\ of $M\cong \CC^{56}$). In fact, when we fix the plane
    $\Pi(0)$ (as we may do up to a projective transformation) the fibre over the point $(x=0,y,y')\in U$ is given by the
    polynomials 
    \begin{equation}\label{eq:F triple_Y_1}
    \begin{aligned}
    G=&x_5Q_5(x_0,\ldots,x_5)+x_4Q_4(x_0,\ldots,x_4)+x_3Q_3(x_0,\ldots,x_3)+\\
    &c(x_0y_1+x_1y_2-x_2)^2(x_0y'_1+x_1y'_2-x_2).
    \end{aligned}
    \end{equation}

    Observe that there are $47={7\choose 2}+ {6\choose 2}+ {5\choose 2}+1$ coefficients involved. All in all, the above
    equations define an affine bundle over the thirteen dimensional base $U$ of rank $47=60-13$. Finally, when $y=y'$ the
    above gives a local description of $\tilde{\cV}^a$ over an open of $\cG_{\Delta}$, isomorphic to $\CC^9\times\CC^2$. The
    varieties $\tilde{\cF}$ and $\tilde{\cV}$ locally over $U$ are the projectivisations of the affine bundles
    $\tilde{\cF}^a_U$ and $\tilde{\cV}^a_U$ respectively. Under the action of $\PGL(5)$, we get that they are locally
    projectivisations of affine bundles.

    To conclude the proof, we note that the above local description gives $\tilde{\cF}$ as, fibrewise, a linear
    subbundle of the trivial bundle $\HH^0(\OO_{\PP^5}(3))\otimes\OO_{\mathcal G}$, which means that it is also globally the
    projectivisation of a vector bundle with tautological bundle the restriction of that of the trivial bundle.
\end{proof}

With notation and local study of the universal families established, we may now proceed to studying the main object of
this section, namely the surface $V$.

\begin{lemma}\label{lem:Vtildesmooth} 
    The locus $\tilde{V}\subset\tilde{F}$ is closed, smooth and of pure dimension $2$. Hence, also its image $V\subset
    F$ is closed and 2-dimensional.
\end{lemma}
\begin{proof}
From Proposition \ref{prop:tcFtcVsmooth} $\tilde{\cV}$ is of dimension $57$ and smooth. As from Corollary
\ref{cor:Vnonempty} the first projection $\mathrm{pr}_1:\tilde{\cV}\to|\OO_{\PP^5}(3)|$ is onto its image (a space of
dimension $55$), we obtain from generic smoothness that the generic fibre is closed, smooth and of pure dimension $2$. 
\end{proof}

\begin{theorem}\label{thm:tVconn}
    The surface $\tilde{V}$ is irreducible, hence also its image $V$ in $F$. 
\end{theorem}
\begin{proof}
Note first that irreducibility will follow from Lemma \ref{lem:Vtildesmooth} if we can show connectedness.
We follow the strategy of \cite[Theorem 6]{barthvdv}, which we shall summarise before delving into the technical
details. We continue using the notation established in various proofs above.

Consider now again the projection \[{\mathrm pr}_1:\tilde{\cV}\to M=\HH^0(\PP^5, \OO_{\PP^5}(3)),\] 
which for a smooth cubic equation $G\in M$ so that $X=\VV(G)$, has fibre
$\tilde{V}$ the strict transform in $\tilde{F}$ of the surface $V$ of triple lines in $X$. The Stein
factorisation of $\pr_1$ is as follows
\[
\xymatrix{
    \tilde{\cV}\ar[dd]^{{\mathrm pr}_1}\ar[dr]^\rho & \\
    & W\ar[dl]^\sigma \\
    M &
}
\]
where $\rho$ has connected fibres, $\sigma$ is finite and $W$ is a normal variety. As the base $M$ is simply connected,
either $\sigma$ is a bijection in which case the theorem is proven, or there is a ramification divisor $D\subset M$ and
a branch divisor $D'\subset W$. As explained above, $\cG_{\Delta}$ is covered by opens isomorphic to
$U=\left\{\left(\ell(x,y), \Pi(x)\right) : (x,y)\in\CC^9\times\CC^2\right\}$, (i.e., \eqref{eq:U} restricted to
$\cG_\Delta$ where we have taken $y=y'$) over which the fibration $\tilde{\cV}$ is isomorphic to the one over $U$, that
is isomorphic to the projectivisation of the affine bundle $\tilde{\cV}^a_U$. We now restrict to the latter and consider
the analogue of the above diagram there.

If we denote by $\tilde{D}\subset\tilde{\cV}^a_U$ an irreducible component of the inverse
image of $D'$, then the differential \[d{\mathrm pr}_1:T_{ \tilde{\cV}^a_U}\to T_M\] cannot be surjective along points of
$\tilde{D}$. We will now arrive at a contradiction by arguing that the locus of points where the differential is not
surjective lies in codimension at least 2, thus ruling out this eventuality. We will prove this by arguing that the
non-surjective locus is of codimension at least two in every fibre of ${\mathrm pr}_2:\tilde{\cV}^a_U\to U\subset \cG_\Delta$, with
notation as in \eqref{eq:pr1tF}, \eqref{eq:cG}. As all fibres are isomorphic, it suffices to prove this over the point
$(\ell_0,\Pi_0)$ as in \eqref{eq:triple0}.

Recall that $\cA \cong \tilde{\cV}^a_U$ and let $\beta : \cA \to M$ the 
composition of ${\mathrm pr}_1$ with this isomorphism. 
The partial derivatives $\partial_{w_{k_0\cdots k_5}},
\partial_{x_{ij}}, \partial_{y_{t}}, \partial_{z}$ for $\sum k_\ell =3, i\in\{0,1,2\},j\in\{3,4,5\}$ and $t\in\{1,2\}$
generate the tangent space $T_{M\times U\times\CC}$, and an explicit computation shows that 
a vector 
\[(\{W_{k_0\cdots k_5}\}, \{X_{ij}\}, \{Y_t\}, Z)\in T_{M\times U\times\CC}\]
is in the subspace $T_\cA|_{{\mathrm pr}_2^{-1}(\ell_0,\Pi_0)}$   at the point $(G_b, c)$ if only if

\begin{align}\label{condit_tangent} 
\begin{split}
\sum_{j=3,4,5} {\partial G_b\over \partial x_j}(s_0,s_1,s_2,0,0,0)(s_0X_{0j}+s_1X_{1j}+s_2X_{2j})\\+
\sum_{k_0+k_1+k_2=3} W_{k_0k_1k_2000} s_0^{k_0}s_1^{k_1}s_2^{k_2}\\ = s_2^2(3c(Y_1s_0+Y_2s_1)-Zs_2).
\end{split}
\end{align}
The map $d \beta  $ sends $(\{W_{k_0\cdots k_5}\}, \{X_{ij}\}, \{Y_t\}, Z)\mapsto(\{W_{k_0\cdots k_5}\})\in T_{G_b}M$.

As Equation \eqref{condit_tangent} imposes no conditions on the $W_{k_0...k_5}$ when $k_0+k_1+k_2\leq 2$, the map
$d\beta $ is then onto at the point $(G_b, (\ell_0,\Pi_0), c)$ if and only if for any tuple
$(\{W_{k_0k_1k_2000}\}:k_0+k_1+k_2=3)$, there exists $(\{X_{ij}\},\{Y_t\}, Z)$ such that Equation \eqref{condit_tangent}
is satisfied.
This translates into the surjectivity of the following linear map 
\[
\phi: \left(\bigoplus_{j=3,4,5}(\CC^3)^*\right) \bigoplus \left(\CC^2\right)^*\bigoplus \CC^* \to \HH^0(\PP^2, \OO(3))  
\]
defined by
\[  
(l_3,l_4,l_5, l, l') \xlongrightarrow{\phi} \sum_{j=3,4,5} {\partial G_b\over \partial x_j}(s_0,s_1,s_2,0,0,0)\, l_j +
s_2^2(c l+l'),
\]
with coordinates $s_0,s_1,s_2$ for $(\CC^3)^*$, $s_0,s_1 $ for $(\CC^2)^*$ and $s_2$ for $\CC^*$.

Choosing the natural basis for the $12$-dimensional space on the left and $10$-dimensional space on the right
respectively, an easy but lengthy computation that we omit allows one to write down the $10\times 12$ matrix
corresponding to $\phi$ explicitly. We note that out of the 46 coefficients involved in defining a $G_b$ which has the pair
$(\ell_0,\Pi_0)$ as a triple line (i.e.,\ of the form \eqref{eq:F triple_Y_1}), only 19 of them arise in the matrix
corresponding to the map $\phi$, i.e.,\ $\phi$ is a $10\times12$ matrix $M$ with 19
variables involved in its entries. Namely for $A_{ijk}$ six $1\times3$ vectors of coefficients of $G_b$ (cf.\
\cite[p.100]{barthvdv}) and $c$, the matrix has the form
\begin{equation}\label{matrix}
M =
\left(
\begin{array}{cccccc|c}
A_{200} &         &         &   &   &   & s_0^3\\
A_{110} & A_{200} &         &   &   &   & s_0^2s_1   \\
A_{101} &         & A_{200} &   &   &   & s_0^2s_2 \\
A_{020} & A_{110} &         &   &   &   & s_0s_1^2\\
A_{011} & A_{101} & A_{110} &   &   &   & s_0s_1s_2 \\
A_{002} &         & A_{101} & c &   &   & s_0s_2^2\\
        & A_{020} &         &   &   &   & s_1^3\\
        & A_{011} & A_{020} &   &   &   & s_1^2s_2\\
        & A_{002} & A_{011} &   & c &   & s_1s_2^2\\
        &         & A_{002} &   &   & 1 & s_2^3
\end{array}
\right)
\end{equation}
where the final column corresponds to the basis of $\HH^0(\PP^2,\OO_{\PP^2}(3))$.

As being of maximal rank is invariant under multiplying by a scalar, to show that the locus where $\phi$ is not
surjective is in codimension at least $2$, it suffices to construct a $\PP^1\subset\PP\HH^0(\PP^2,\OO_{\PP^2}(3))$, over
every point of which $\phi$ is surjective - as if it were in codimension $\leq1$, then this $\PP^1$ would necessarily
meet this locus of non-surjectivity. In other words, to conclude we need to choose two linearly independent vectors in
$\CC^{19}$, i.e., entries for the $A_{ijk}$ and $c$ above, so that the plane they span corresponds to a 2-dimensional
subspace of the space of coefficients of cubic polynomials of the form \eqref{eq:F triple_Y_1}, at all of which the
matrix defined by $\phi$ has rank 10. To this aim, we let $\kappa,\lambda$ be coordinates of such a $\PP^1$ and consider
the following matrix (chosen essentially arbitrarily to be in the form \eqref{matrix})

\[
M(\kappa,\lambda) =
\left(\begin{array}{cccccccccccc}
\kappa & 0 & \lambda & 0 & 0 & 0       & 0 & 0 & 0       & 0 & 0 & 0 \\
0 & 0 & \kappa & \kappa & 0 & \lambda & 0 & 0 & 0       & 0 & 0 & 0 \\
\lambda & \lambda & \kappa &  0 & 0 & 0      & \kappa & 0 & \lambda & 0 & 0 & 0 \\
0 & \lambda & \kappa & 0 & 0 & \kappa & 0 & 0 & 0       & 0 & 0 & 0 \\
\lambda & \kappa & 0 & \lambda & \lambda & \kappa & 0 & 0 & \kappa & 0 & 0 & 0 \\
0 & 0 & 0 & 0 & 0 & 0       & \lambda & \lambda & \kappa & \kappa & 0 & 0 \\
0 & 0 & 0       &  0 & \lambda & \kappa& 0 & 0 & 0       & 0 & 0 & 0 \\
0 & 0 & 0       & \lambda & \kappa & 0 & 0 & \lambda & \kappa & 0 & 0 & 0 \\
0 & 0 & 0       & 0 & 0 & 0 & \lambda & \kappa & 0 & 0 & \kappa & 0 \\
0 & 0 & 0       &  0 & 0 & 0      & 0 & 0 & 0 & 0 & 0 & 1 
\end{array}\right).
\]
It is now not hard to check by a Gr\"obner basis computation using Macaulay2 that the $10\times10$ minors of the above
matrix have no common solution for $(\kappa,\lambda)\neq(0,0)$ which concludes the proof.
\end{proof}

\begin{remark}
In fact the proof of Theorem \ref{thm:tVconn} gives that the locus of triple lines in any cubic fourfold $X$ (i.e.,\ not
just smooth) is connected. If $X$ is general, then $V$ is irreducible even. 
\end{remark}

We recall that $\cU_k$ denotes the universal subbundle on $G(k,6)$. We denote by $\pi_i$, $i=1,2,3$, the projections of
$ \G(2,6)\times \G(2,6)\times\G(3,6) $ to the three factors and by $\pi_{13}$ the projection to the product of the first
and third factor. We recall the definitions of $\cG=\{(\ell, \ell',\Pi) : \ell', \ell\subset\Pi\} \subset \G(2,6)\times
\G(2,6)\times\G(3,6),\; \tilde{\cF} \subset |\cO_{\bP^5}(3)| \times \cG,\tilde{\cV}\subset \tilde{\cF}$ from
\eqref{eq:cG}, \eqref{eq:pr1tF}. We have the projections ${\rm pr}_1 : \tilde{\cF} \to |\cO_{\bP^5}(3)|$ and ${\rm
pr}_2: \tilde{\cF} \to \cG $. The fibre $\tilde{F}$ of $\tilde{\cF} $ under the first projection maps to $\cG$ under the
second projection. Following Equation \eqref{eq:pr1tF}, we explained that the fiber under ${\rm pr}_1$ over a general
point $X$ is the graph of the Voisin map, which we saw is isomorphic to the blow up $\tF$ of $F$ at $S$. From now
on, we denote either one of these by $\tF$.  We keep the notation $\pi_i$, $i=1,2,3$, for the composite maps from
$\tilde{F}$ to the three factors of the $\G(2,6)\times \G(2,6)\times\G(3,6)$. 

\begin{proposition}\label{prop:normalbundletV}
    The subvariety $\tV\subset\tF$ is the vanishing locus of a section of the rank 2 bundle
    \[B=\pi_2^*\cU_2^\vee \otimes \bigwedge^3\pi_3^*\cU_3\otimes \bigwedge^2\pi_1^*\cU_2^\vee.\]
    In particular, $N_{\tV/\tF}\cong B|_{\tV}$ so for $H=\pi^*H_F$ the pullback of the Pl\"ucker line bundle to
    $\tilde{F}$, we have an isomorphism of line bundles
    \begin{align*}
        K_{\tV} &= 3H_{\tV}.
    \end{align*}
\end{proposition}
\begin{proof}
We note that $\cG_\Delta$ is isomorphic to the locus $\mathcal{D}=\G(2,\cU_3)=\{(\ell,\Pi) : \ell\subset\Pi\}\subset
\G(2,6)\times\G(3,6)$. Let $\gamma : \G(2,6)\times\G(3,6) \to \G(2,6)$ and $\beta : \G(2,6)\times\G(3,6) \to
\G(3,6)$ the projections. The tautological sequence of the Grassmann bundle $\mathcal{D}=\G(2,\cU_3)$ is
\begin{equation}\label{eq:tautol_X}
0 \to \gamma^*\cU_2 \to \beta^*\cU_3 \to Q \to 0,
\end{equation} 
with $Q=\bigwedge^3\beta^*\cU_3 \otimes (\bigwedge^2 \gamma^*\cU_2)^{\vee}$ a line bundle. The variety $\cG$ is the
Grassmann bundle $h:\cG= \G(2,\beta^*\cU_3)\to \mathcal{D}$ and 
note that $h$ is the restriction of $\pi_{13}$ to $\cG$. The tautological sequence on $\cG$ is
\begin{equation}\label{eq:tautol_G}
0 \to \pi_2^*\cU_2 \to h^*\beta^* \cU_3=\pi_3^*\cU_3 \to Q_{\cG} \to 0,
\end{equation} 
with $Q_{\cG}= \bigwedge^3h^* \beta^*\cU_3 \otimes (\bigwedge^2 \pi_2^*\cU_2)^{\vee}$ a line bundle. 

The inclusion $\cG_\Delta\hookrightarrow\cG$ comes from a section $\delta$ to the projection $\pi_{13}$, with $\delta
(\ell, \Pi)=(\ell, \ell, \Pi)$. This section is determined (see \cite[B.5.7]{fulton}) by the rank two subbundle
$\delta^*\pi_2^*\cU_2$ of $\beta^*\cU_3$ on $\mathcal{D}$. Since $\pi_2 \delta =\gamma $ we have
$\delta^*\pi_2^*\cU_2=\gamma^*\cU_2$. Note then that $\delta^*Q_{\cG}=Q$. By \cite[B.5.8]{fulton} we have that the
relative tangent bundle of $h$ equals 
\[
T_h=\cU_{\cG}^{\vee} \otimes Q_{\cG}
\]
and then the relative normal bundle of $\cG_{\Delta}\cong \mathcal{D} $ in $\cG$ equals (see, \cite[B.7.3]{fulton})
\begin{equation}\label{eq:normalGD}
N_{\cG_\Delta/\cG} \cong \delta^* \pi_2^*\cU_2^{\vee} \otimes \delta^* Q_{\cG} \cong
\gamma^*\cU_2^\vee \otimes Q \cong \gamma^*\cU_2^\vee \otimes\bigwedge^3\beta^*\cU_3\otimes
(\bigwedge^2 \gamma^*\cU_2)^\vee.
\end{equation}
 
Since now the section $\delta$ is determined by the subbundle $\gamma^*\cU_2$ of $ \beta^*{\cU}_3$, we have as in
\cite[B.5.6]{fulton}, (adapted to the case of Grassmann bundles) that the composite of the inclusion $\pi_2^*\cU_2 \to
h^*\beta^*{\cU}_3$ with the pull back $h^* \beta^*{\cU}_3\to h^*Q$ of the canonical map $ \beta^*{\cU}_3 \to Q$ gives a
map $\pi_2^*\cU_2 \to h^*Q= \bigwedge^3\pi_3^*\cU_3\otimes(\bigwedge^2\pi_1^*\cU_2)^\vee $, which corresponds to a
section of the rank 2 bundle $\pi_2^*\cU_2^{\vee} \otimes h^*Q$. The zero locus of this section is the image
$\cG_{\Delta}$ of the section $\delta$. 

From \cite[\href{https://stacks.math.columbia.edu/tag/0473}{Tag 0473}]{stacks-project} we have that
$N_{\tilde{\cV}/\tilde{\cF}}=\pr_2^*N_{\cG_\Delta/\cG}$. Restricting to a fibre $\tF\subset\tilde{\cF}$ over a general
smooth $X\in|\OO(3)|$, we have a surjection $N_{\tilde{\cV}/\tilde{\cF}}^\vee|_{\tV}\to N_{\tilde{V}/\tilde{F}}^\vee$,
and as both are locally free of rank two (the latter from Lemma \ref{lem:Vtildesmooth} and Theorem \ref{thm:tVconn}),
they must be isomorphic so we compute $N_{\tV/\tF}$ as in the statement. We now use the following result,
noting that the second equality, although proven numerically, also holds at the level of line bundles since
$\Pic^0(\tF)=0$ (as $\HH^1(F, \OO_F)=0$).
\begin{lemma}(\cite[Remark 2.2.19]{huybrechts}, \cite[Proposition 6]{amerik})\label{lem:huyamer}
We have the following equalities of line bundles on $\tF$
\[
\begin{aligned}
\pi_3^*\rm{c}_1(\cU_3)&=-3H+E_S, \\
\pi_2^*\rm{c}_1(\cU_2)&=-7H+3E_S. 
\end{aligned}
\]
\end{lemma}
On the other hand, since $\pi_1=\pi_2$ when we restrict to $\tilde{V}$, we have that
$\pi_1^*(\cU_2)=\pi_2^*(\cU_2)$ on $\tilde{V}$ hence $-H=-7H+3E_S$, i.e., 
\begin{align}\label{eq:HES}
E_S=2H \text{ on }\tilde{V}.
\end{align}
By plugging the above in the formula for $\det B$ we get $\wedge^2 N_{\tV/\tF} = (3H-E_S)_{\tV}=H_{\tV}$.  Finally, the
isomorphism $K_{\tV}=3H_{\tV}$ follows from the fact that $K_{\tF}=E_S$ as $\tF$ is the blowup of $F$ at $S$.
\end{proof}

\begin{theorem}\label{thm:class of V}
     The class of $V$ in $\CH^2(F)$ is given by
    \[[V]=21\mathrm{c}_2(\cU_F).\]
\end{theorem}
\begin{proof} 
    We have that ${\mathrm c}_1(\pi_1^* \cU_2)=-H$. Recall that $ \bigwedge^3\pi_3^*\cU_3\otimes
    (\bigwedge^2\pi_1^*\cU_2)^\vee=h^*Q$. In the following we repeatedly use the formulas of Lemma \ref{lem:huyamer}. We
    have $\mathrm{c}_1(h^*Q)=\mathrm{c}_1(\pi_3^* \cU_3)-\mathrm{c}_1(\pi_1^*\cU_2)=-2H+E_S$ so that
    \begin{align*}
     [\tilde{V}] &=\mathrm{c}_2(\pi_2^*\cU_2^\vee \otimes h^*Q)=\mathrm{c}_2(\pi_2^*\cU_2)-\mathrm{c}_1(\pi_2^*\cU_2)\mathrm{c}_1(h^*Q)+\mathrm{c}_1^2(h^*Q)\\
    &=\mathrm{c}_2(\pi_2^*\cU_2)+(7H-3E_S)(-2H+E_S)+(-2H+E_S)^2\\
    &=\mathrm{c}_2(\pi_2^*\cU_2)-10H^2+9HE_S-2E_S^2.
    \end{align*}
    From \eqref{eq:tautol_G} we get $\mathrm{c}_2(\pi_2^*\cU_2)= \mathrm{c}_2(\pi_3^*
    \cU_3)-\mathrm{c}_1(\pi_2^*\cU_2)\mathrm{c}_1(Q_{\cG})$, with $\mathrm{c}_1(Q_{\cG})= \mathrm{c}_1(\pi_3^*
    \cU_3)-\mathrm{c}_1(\pi_2^*\cU_2)=4H-2E_S$, hence 
    \[
    \mathrm{c}_2(\pi_2^*\cU_2)= \mathrm{c}_2(\pi_3^* \cU_3)+ 28H^2-26HE_S+6E_S^2.
    \]
    From \eqref{eq:tautol_X} we get that $\mathrm{c}_2(\pi_3^* \cU_3) = \mathrm{c}_2(\pi_1^*\cU_2) +
    \mathrm{c}_1(\pi_1^*\cU_2)\mathrm{c}_1(h^*Q)$, with $\mathrm{c}_1(h^*Q)=\mathrm{c}_1(\pi_3^*
    \cU_3)-\mathrm{c}_1(\pi_1^*\cU_2)=-2H+E_S$, hence $\mathrm{c}_2(\pi_3^* \cU_3) =\mathrm{c}_2(\pi_1^*\cU_2)+2H^2-HE$.
    Therefore 
    \[
    \mathrm{c}_2(\pi_2^*\cU_2)=\mathrm{c}_2(\pi_1^*\cU_2)+30H^2-27HE_S+6E_S^2
    \]
    and then
    \begin{equation} \label{eq:classtV}
     [\tilde{V}]= \mathrm{c}_2(\pi_1^*\cU_2) +20H^2+4E_S^2-18HE_S
    \end{equation}
    from which we may now conclude, by pushing forward via $\pi$ and using Theorem \ref{thm:amerik and osy}
    \begin{align*}
    [V]&=\pi_*[\tilde{V}] \\
    &=20H_F^2-4[S]+\mathrm{c}_2(\cU_F) \\
    &=21\mathrm{c}_2(\cU_F).\qedhere
    \end{align*}
\end{proof}

To summarise the above, we conclude the following.

\begin{corollary}\label{cor:classV}
    The locus $V\subset F$ of triple lines is an projective irreducible surface of class $21\mathrm{c}_2(\cU_F)$. 
\end{corollary}

In order to analyse singularities of $V$ we will now need to study the intersections 
\begin{align*}
C&=S\cap V\subset F,\\
\tilde{C}&=E_S\cap\tilde{V}\subset\tilde{F}. 
\end{align*}

\begin{corollary}\label{cor:Cnonempty}
    We have $C\neq\emptyset$.
\end{corollary}
\begin{proof}
  Theorems \ref{thm:class of V}, \ref{thm:amerik
    and osy} and Lemma \ref{lem:intersection theory} give  
    \[[S][V] = 5(H_F^2-\mathrm{c}_2(\cU_F^{\vee})) \; 21\mathrm{c}_2(\cU_F) = 1890,\]
    so that in particular $S\cap V$ cannot be empty.
\end{proof}

\begin{proposition}\label{prop:Ctildesmooth}
    The variety $\tilde{V}$ intersects $E_S$ transversely at a smooth pure 1-dimensional variety $\tilde{C}$. 
\end{proposition}
\begin{proof}
Denote by $\cE\subset\tilde{\cF}$ the inverse image of the universal type two locus $\cS\subset |\OO_{\PP^5}(3)|\times
\G(2,6)$ defined at the beginning of this section.  The fibre ${\mathrm pr}_1^{-1}(X)$ restricted to $\cE$ is $E_{S_X}$
the exceptional divisor of the blow up, whose fibres over the type two locus $S_X$ of $X$ parametrise 2-planes tangent
along a line in $S_X$. We show now that all fibres of $\mathrm{pr}_2: \cE \to \cG$ are isomorphic. Indeed, a projective
transformation takes a second type line of a smooth cubic to a second type line since for a smooth cubic $X=\VV(G)$,
$\ell$ being of second type is equivalent to 
\[
\langle \partial_iG|_\ell\rangle_{i=0,\ldots,5} \subset\HH^0(\ell, \OO(2))
\]
being 2-dimensional (instead of 3-dimensional which is the generic case). Keeping the notation established in the
beginning of this sections and  restricting at the fibre over $(x=0, y=0, y')$, i.e. $\Pi=\Pi_0$, $\ell=\ell_0$ as in
\eqref{eq:triple0},  then using the equation of  $G$ from \eqref{eq:F triple_Y_1}, we see that
\begin{align*}
\langle \partial_iG|_{\ell_0}\rangle_{i=0,\ldots,5} = \langle &0,0,0, Q_3(x_0,x_1,0,0), \\&Q_4(x_0,x_1,0,0,0),
Q_5(x_0,x_1,0,0,0,0)\rangle.
\end{align*}
Since three quadrics in two variables are linearly dependent if and only if the determinant of the corresponding
$3\times3$ matrix vanishes, we see that this imposes one extra condition. We also conclude that all the fibres of the
map ${\mathrm pr}_2:\cE\to\mathcal{G}$ are isomorphic to a determinant cubic hypersurface (inside the fibre
$\PP^{46}$ of the second projection of $\tilde{\cF}$), which is smooth outside the codimension 4 locus of the generic
determinantal variety $M_1(3,3)$. This singular locus maps under the first projection to the locus of singular cubics
since in this case the cubic is singular at a point of the line. The degeneracy locus $M_1(3,3)$ can be constructed
universally and as it maps into the locus of singular cubics in $|\OO_{\PP^5}(3)|$ (cf.\ the proof of \cite[Proposition
2.2.13]{huybrechts}), we obtain also that
$\cE_{\mathrm{sm}}={\mathrm pr}_1^{-1}(|\OO_{\PP^5}(3)|_{\mathrm{sm}})\subset\cE$ is smooth.

The intersection $\tilde{\cV} \cap\cE $ is simply the restriction of $\cE$ to the locus $\cG_\Delta$ under $\mathrm{pr}_2$.
Restricting to the locus of smooth cubics, we have a diagram
\[
\xymatrix{
    \tilde{\cV} \cap\cE_{\mathrm{sm}}\ar[r]^-{\mathrm{pr}_1}\ar[d]^{\mathrm{pr}_2} & 
    |\OO_{\PP^5}(3)|_{\mathrm sm} \\
    \cG_\Delta &
}
\]
where any fibre over $\mathcal{G}_\Delta$ is the fixed determinantal cubic hypersurface in $\PP^{46}$ from above. The
restricted map ${\mathrm pr}_1$ is onto from Corollary \ref{cor:Cnonempty}. In particular, as the base $\cG_\Delta$ is
smooth and the intersection above is just the restriction to $\cG_\Delta$, so also the total space $\tilde{\cV}
\cap\cE_{\mathrm{sm}}$ is smooth. By generic smoothness, also
$\mathrm{pr}_1^{-1}(\VV(G))$ is smooth for generic $\VV(G)\in |\OO_{\PP^5}(3)|_{\mathrm{sm}}$ and must have
pure 1-dimensional fibres. This implies that $\tilde{V}$ meets $E_S$ transversely for a fixed general cubic $\VV(G)$. 
\end{proof}

We note that the above proof does not conclude that $\tC$ is irreducible, yet this will be proven later in Lemma
\ref{lem:tVandC}.

We consider now the $12$-dimensional family 
\begin{align}\label{eq:cH}
{\mathcal H}=\{(\ell, H) \text{ with } \ell \subset H\}\subset \G(2,6)\times \G(4,6),
\end{align}
which is a $\G(2,4)$-fibration over $\G(2,6)$. Since any two such pairs differ by a projective transformation, the
family of smooth cubic fourfolds that contain $\ell$ as a line of second type with corresponding 3-plane $H$ are all
isomorphic as we vary $\ell$ and $H$. We may then assume $\ell_0, \ H_0$ are as in \eqref{eq:triple0}. The smooth cubics
in $\bP^5$ which contain $\ell_0$ as a line of second type and whose corresponding tangent 3-plane is $H_0$ form an open
set inside the family of cubics given by the following equation
\begin{align}\label{eqtype2full}
\begin{split}
F=&\ (c_0x_4+d_0x_5)x_0^2+(c_{01}x_4+d_{01}x_5)x_0x_1+(c_1x_4+d_1x_5)x_1^2+\\
& x_0Q_0(x_2,x_3,x_4,x_5)+x_1Q_1(x_2,x_3,x_4,x_5)+ P(x_2,x_3,x_4,x_5),
\end{split}
\end{align}
with $Q_0,Q_1$ quadrics and $P$ a cubic (cf.\ Equation \eqref{eqtype2} where 
one member of the above family was chosen). The above family has dimension $45$ and thus the variety
\begin{align*}
\left\{({\VV}(G),(\ell, H)):\ell \subset {\VV}(G) \mbox{ has tangent 3-plane } H \right\} 
\subset |{\mathcal O}_{\bP^5}(3)|_{\text{sm}} \times {\mathcal H} 
\end{align*}
is irreducible of dimension $57$ and smooth - note here that this variety is isomorphic to $\cS$ introduced above since
every type two line has a unique triple tangent hyperplane.

The pencil of 2-planes $\Pi_{\lambda, \mu}$ inside $H_0\cong \PP^3$ (with coordinates $x_0,x_1,x_2,x_3$) which contain
$\ell_0$ is given by $\mu x_2 - \lambda x_3=0$. The cubic surface ${\VV}(G)\cap H_0$ is given by the equation 
\[ x_0Q_0(x_2,x_3)+x_1Q_1(x_2,x_3)+P(x_2,x_3)=0, \] 
where $Q_0(x_2,x_3)=Q_0(x_2,x_3,0,0)$ etc.\ The intersection of
$\Pi_{\lambda, \mu}$ with ${\VV}(G)$ is then given by the equation (without loss of generality we assume that $\lambda\neq0$ in
the following)
\begin{align*}
x_0 x_2^2Q_0(\lambda, \mu)+&x_1x_2^2Q_1(\lambda, \mu)+x_2^3\lambda^{-1}P
(\lambda, \mu)\\&=x_2^2\, [x_0 Q_0(\lambda, \mu)+x_1Q_1(\lambda, \mu)+
x_2\lambda^{-1}P(\lambda,\mu)]=0,
\end{align*}
and the solution $x_2^2=0$ is the double line $\ell_0$ whereas $x_0 Q_0(\lambda, \mu)+x_1Q_1(\lambda, \mu)+
x_2\lambda^{-1}P(\lambda,\mu)=0$ is the residual line $\ell_{\lambda, \mu}$ (the equations take place in the plane
$\Pi_{\lambda, \mu}$). Recall that when $X={\VV}(G)$ is  general then it does not contain a 2-plane. Hence, in this
case, for any $[\lambda, \mu]$ the last equation  is not the null equation, otherwise  $\Pi_{\lambda, \mu} \subset X$.
The intersection point of the residual line with $\ell_0$ is given by 

\begin{align}\label{eq:Q01}
x_0 Q_0(\lambda, \mu)+ x_1Q_1(\lambda, \mu)=0.
\end{align} 
Finally, $Q_0, Q_1$ have a common root $[\lambda_0, \mu_0]$ if and only if there exists a 2-plane, namely
$\Pi_{\lambda_0,\mu_0}$, which intersects $X$ at $x_2^3$, in which case $\ell_0\in S\cap V$. By a dimension count, none
of them is identically zero and they cannot have a common double root as $X$ is general.

We note the following corollary of the above discussion, which will be used in the next section.

\begin{corollary}\label{cor:2to1}
    For $s\in S\setminus(S\cap V)$, the restriction of the map \[\psi_{E_S}:E_S\dashrightarrow W\] to
    $E_{S,s}\cong\PP^1$ is a two-to-one morphism. 
\end{corollary}
\begin{proof}
    The map $\psi_{E_S}$ is not defined on the fibres of $E_S$ over $V\cap S$ and from the discussion above this
    corresponds to the case in which $Q_0,Q_1$ have a common root. If this is not the case, then given a point
    $[x_0:x_1]\in \ell_0$ there are two pairs $[\lambda,\mu]$ (counted with multiplicity) which satisfy Equation
    \eqref{eq:Q01} which proves the claim.
\end{proof}

Viewed geometrically, given $p\in \ell$, $[\ell]\in S$ general, the line $\ell$ corresponds to the node of $C_p$, the
nodal genus four curve which parametrises lines through $p\in\ell$ (see Lemma \ref{lem:sings Cx}). As a generic nodal
genus four curve possesses two $g^1_3$'s and the fibre through the node of each one of these defines a conjugate point,
the above proposition (and Proposition \ref{prop:ram pts}) give that we have two conjugate lines intersecting $\ell$ at
the point $p$.

\begin{lemma}\label{lem:tVandC}
    The smooth irreducible surface $\tV$ does not contain any fibres of $E_S\to S$. Hence $\tC$ is smooth and
    irreducible and thus $C$ is an irreducible curve whose normalisation is $\tC$. The locus of $s\in S$ over which
    $E_{S,s}$ meets $\tV$ in two points is a finite subset of $C=S\cap V$.
\end{lemma}
\begin{proof}
    From Proposition \ref{prop:Ctildesmooth}, $\tilde{C}=E_S\cap\tilde{V}$ is smooth 1-dimensional. Note now that the
    above analysis gives further information on $\tilde{V}$. First, $\tilde{V}\subset\tilde{F}$ cannot contain an entire
    fibre of $E_S\to S$ since then $Q_0(\lambda,\mu), Q_1(\lambda,\mu)$ of Equation \eqref{eq:Q01} would both be
    identically zero, which cannot happen for a general fourfold. From \eqref{eq:HES} we get that $\tC$ is then an ample
    divisor of $\tV$, since $\tV\to V$ is a finite morphism. Hence it is connected from Bertini, and from smoothness of Proposition
    \ref{prop:Ctildesmooth}, it is also irreducible. In other words $C=S\cap V=\pi(\tC)$ is irreducible and $\pi:\tC\to
    C$ is the normalisation morphism.  Furthermore, $Q_0$ and $Q_1$ can have at most two common roots, and these
    correspond to the number of intersections points of $\tilde{V}$ with a fibre of $E_S$. As from Corollary
    \ref{cor:Cnonempty} $V\cap S\neq\emptyset$ for a general cubic, having one common root imposes one condition in the
    fibre $S$ of $\mathcal{S}$ over $X=\VV(G)\in|\OO_{\PP^5}(3)|$, whereas having two common roots is a codimension two
    condition in $S$, i.e., a finite number of points on $C$. 
\end{proof}

\begin{lemma}\label{lem:classCgenustC}
    The class of $C$ is given by \[[C]=\frac{35}{2}H_F^3\in \HH_2(F,\ZZ)\] (or $6H_S\in\HH^2(S,\QQ)$). We have 
    $p_a(\tilde{C})=4726,\ p_a(C)=8506$.
\end{lemma}
\begin{proof}
    By the definition of these loci, we have $\tC=E_S|_{\tV}$. From \eqref{eq:classtV} it follows that
    \[ [\tilde{C}]=[\tilde{V}].E_S = 20H^2E_S-18HE_S^2+4E_S^3+E_S\mathrm{c}_2(\pi_1^*\cU_2). \] 
    To compute the terms, we use the standard diagram
    \[
        \xymatrix{
        E_S=\PP(N_{S/F}) \ar@{^{(}->}[r]^-j \ar[d]^{\pi'} & \tF \ar[d]^{\pi} \\
        S \ar@{^{(}->}[r]^-i & F.
        }
    \]    
    Let $\xi$ be the tautological bundle of $E_S=\PP (N_{S/F})$. Then $j^*E_S=\xi$, which implies that $\pi_*E_S^k
    =i_*\pi'_{*}\xi^{k-1}$, and let $\xi^2-\mathrm{c}_1(\pi'^*N_{S/F}) \xi +\mathrm{c}_2(\pi'^*N_{S/F})=0$ the standard
    relation. We then have $\pi_*(H^2E_S)=\pi_*(\mathrm{c}_2E_S)=0, \pi_*HE_S^2=-H_S$ and
    $\pi_*E_S^3=-\mathrm{c}_1(N_{S/F})$. As $K_F=0$, we obtain $\mathrm{c}_1(N_{S/F})=K_S=3H_S$ from Theorem \ref{thm:amerik
    and osy}. Then $[C]=\pi_*[\tilde{C}]=6H_S$. Theorem \ref{thm:amerik and osy} and Lemma \ref{lem:intersection theory}
    imply the result.

    To compute the genus of $\tilde{C}\subset\tV$, note that $K_{\tV}=3H_{\tV}$ from Proposition \ref{prop:normalbundletV}
    and, as mentioned in \eqref{eq:HES}, we have that $E_S|_{\tV}=2H|_{\tV}$. Therefore
    $[\tC(\tC+K_{\tV})]= 10H^2_{\tV}=10H[\tV] $. From \eqref{eq:classtV}, we conclude that
    \[ [\tC(\tC+K_{\tV})] =10( \mathrm{c}_2(\pi_1^*\cU_2)H^2 +20H^4+4H^2E_S^2-18H^3E_S). \] 
    Using Theorem \ref{thm:amerik and osy} and Lemma \ref{lem:intersection theory} we get $
    \mathrm{c}_2(\pi_1^*\cU_2)H^2=45$, $H^4=108$, $H^2E_S^2=-[S]H_F^2=-315$, and $H^3E_S=H_F^3\pi_*[S]=0$. Hence
    $[\tC(\tC+K_{\tV})] =9450$ implying the first genus computation. Similarly, computing now via adjunction for $C\subset
    S$ we have from the above that $C(C+K_S)=54H_S^2=54H_F^2[S]=54\cdot 315=17010$ so that $p_a(C)=8506$.
\end{proof}

In the following, we say that a point $x$ in a variety $X$ has \textit{normal crossing} singularities if 
\[\widehat{\OO}_{X,x} \cong \CC[[x_1,\ldots,x_n]]/(x_1\cdots x_r).\]

\begin{proposition}\label{prop:cCnormalcrossings}
    The universal intersection curve $\mathcal{C}=\mathcal{S}\cap\mathcal{V}\subset\cF$ is irreducible and there is an
    open subset $U\subset|\OO(3)|$ so that $\cC|_U$, the restriction to the inverse image of $U$, has normal crossings
    singularities and its normalisation $\cC|_U^\nu\to\cC|_U$ identifies pairs of points of a divisor.
\end{proposition}
\begin{proof}
    From the analysis around Equation \eqref{eq:cH}, the second projection of the universal type two locus
    $\cS\subset|\OO(3)|_{\mathrm{sm}}\times\cH\to\cH$ is a $\PP^{45}$-fibration over the locus $\cH$ of pairs $(\ell,H)$
    where $\ell\subset H\cong\PP^3\subset\PP^5$ (see \eqref{eq:cH}), and the line $\ell$ will be a type two triple line
    (i.e., in $S\cap V$) if and only if for $G$ of type \eqref{eqtype2full}, the plane quadrics $Q_0,Q_1$ have a common
    root (see Lemma \ref{lem:tVandC}). In particular the fibre of $\cC$ over $\cH$ in $\PP^{45}$ - i.e., the locus
    parametrising cubics $G$ for which a fixed pair $(\ell_0,H_0)\in\cH$ is a triple type two line - is a cone over the
    corresponding resultant $R$ in 6 variables, namely the coefficients of the $Q_i$. Note also that the induced
    $\mathrm{PGL}_5$ action on points of $\cH$ is transitive, as the property of being a triple type two line is
    preserved under this action.

    Denoting these coefficients $x_1,x_2,x_3$ (resp.\ $y_i$) of $Q_1$ (resp.\ $Q_2$), the resultant is given by
    \[ R=\VV((x_1y_3-x_3y_1)^2-(x_1y_2-x_2y_1)(x_2y_3-x_3y_2)). \]
    Since $R$ is irreducible, the universal locus $\cC$ must also be irreducible.

    Observe that the family of vertices of the cones over the resultant in each fibre $\PP^{45}$ are in high codimension
    when mapped to $|\OO(3)|$, so as our cubics are assumed general, we may ignore this locus, and we denote by
    $\widehat{R}\subset\PP^{45}$ its complement in the cone, which is of dimension 44. 
    \begin{lemma}\label{lem:Rnc}
        The resultant $R$ generically has normal crossing singularities, resolved by a degree 2 normalisation. 
    \end{lemma}
    \begin{proof}
    The normalisation $R^{\nu}$ of the resultant $R$ is given by the vanishing locus of the polynomials
    \[ \begin{aligned}
    & wx_3-w'y_3,\\
    & wx_2-w'y_2+x_3y_1-x_1y_3,\\
    & wx_1-w'y_1+x_2y_1-x_1y_2,\\
    & w'^2-w'x_2+x_1x_3,\\
    & ww'-w'y_2+x_3y_1,\\
    & w^2-wy_2+y_1y_3
    \end{aligned} \]
    in $\CC[w,w',x_1,\ldots,y_3]$. The singularities of the normalisation are in codimension 4 so we may ignore them as
    we are working generically. The singular locus of the resultant is the smooth rank $\leq 1$ determinantal locus
    given by the matrix
    \[ \begin{pmatrix}
    x_1 & x_2 & x_3 \\
    y_1 & y_2 & y_3
    \end{pmatrix}.  \]
    It consists of points of the form
    \[ [x_1: x_2: x_3: tx_1: tx_2:tx_3] \mbox{ or } [ty_1:ty_2:ty_3:y_1:y_2:y_3].  \]
    and over these points the normalisation has, generically, two points, e.g., over the first type of point the two
    points are \[[tw'_i:w'_i: x_1:x_2:x_3:tx_1:tx_2:tx_3],\ i=1,2,\] where $w'_i$ are the roots of the equation
    $w'^2-w'x_2+x_1x_3$. It has one point over the locus $x_2^2=4x_1x_3$ (where the quadric has a double root, i.e., $w'=x_2/2$).

    A calculation shows that, after performing row transformations, the Jacobian matrix at the point $[tw': w':
    x_1: x_2: x_3: tx_1: tx_2:tx_3]$ is in all cases equivalent to the matrix (we omit zero rows)
    \[J=\begin{pmatrix}
                  0 & 0 & tx_3^2 & -tx_3w' & tw'^2 & -x_3^2 & x_3w' & -w'^2 \\
                  0 & 2w'-x_2 & x_3 & -w' & x_1 & 0 & 0 & 0 \\
                  2w'-x_2 & 0 & 0 & 0 & 0 & x_3 & -w' & x_1 
    \end{pmatrix}, \]
    except when $(w',x_3)=(0,0)$, in which case we get the matrix 
    \[J=\begin{pmatrix}
                  0 & 0 & -tx^2_2 & tx_1x_2 & -tx_1^2 & x^2_2 & -x_1x_2 & x_1^2 \\
                  0 & -x_2 & 0 & 0 & x_1 & 0 & 0 & 0 \\
                  -x_2 & 0 & 0 & 0 & 0 & 0 & 0 & x_1 
    \end{pmatrix}.  \]

    The above matrices have, outside the locus $x_2^2=4x_1x_3$, the first $3\times 2$ minor of rank $2$. Hence, outside
    the above locus, the differential of the map  $p: R^{\nu} \to R$ (which is the restriction of the projection $p :
    \PP^7 \dasharrow \PP^5$ to the last six coordinates) at the above point is injective: the tangent space is the
    linear space of vectors $v$ with $Jv=0$. The kernel of the differential consists of those vectors with the last six
    coordinates zero, but in this case the first two coordinates are also zero.
     
    Moreover, locally the images of the two sheets of the normalisation, where the points $[tw'_i:w'_i:
    x_1:x_2:x_3:tx_1:tx_2:tx_3]$, $i=1,2$, belong, are smooth hypersurfaces in $\PP^5$. The corresponding perp vectors
    of those images at the points $[x_1:x_2:x_3:tx_1:tx_2:tx_3]$, $i=1,2$, when $(w',x_3)\neq (0,0)$ are given by 
    \[ \langle \; tx^2_3, \;\; -tx_3w', \;\; tw'^2, \;\; -x_3^2, \;\; x_3w', \;\; -w'^2 \; \rangle \]
    with $w'=w'_i$ and by 
    \[ \langle
     \-tx^2_2 ,\;\; tx_1x_2, \;\; -tx_1^2 ,\;\; x^2_2 ,\;\; -x_1x_2 ,\;\; x_1^2
    \rangle \]
    when $(w',x_3)=(0,0)$. These are not parallel in all cases and hence we have two branches with normal crossing
    there.
    \end{proof}
     
    As $\cC\to\cH$ is generically isotrivial, a standard argument with the $\mathrm{Isom}$-scheme (see, e.g.,
    \cite[p.3]{buium}) shows that there exists an open neighbourhood $U$, and a finite \'etale cover $U'\to U$ so that
    the pullback of $\cC$ (minus the vertices of the cones as above) is isomorphic to
    \[U'\times \widehat{R}.\]
    The transitivity of the $\PGL_5$-action on $\cH$ implies that any point in $\cH$ has a finite \'etale
    neighbourhood which trivialises $\cC$. Clearly $U'\times \widehat{R}$ has normal crossing singularities outside a
    codimension 2 locus from Lemma \ref{lem:Rnc}. As being normal crossings can be checked \'etale locally, we deduce
    that $\cC$ has normal crossing singularities outside a codimension 2 closed subset too. By a dimension count, the
    non-normal-crossing locus cannot surject onto $|\OO(3)|_\mathrm{sm}$, so after removing its image, we may assume
    there is an open in $|\OO(3)|$ over which the universal family $\cC$ has normal crossings.
\end{proof}

\begin{theorem}\label{theo:nodesC}
    The curve $C$ is irreducible and has 3780 nodes and no other singular points.
\end{theorem}
\begin{proof}
    From Corollary \ref{lem:classCgenustC}, as the difference between the genus of $C$ and its normalisation $\tC$ is
    exactly 3780 (in particular positive), we obtain that $C$ must be singular for a generic cubic fourfold. From Lemma
    \ref{lem:tVandC}, $C$ is irreducible. To conclude the result, it suffices thus to show that $C$ is nodal.

    To this aim, we note that Proposition \ref{prop:cCnormalcrossings} proves that $\cC$ has normal crossings over an
    open $U\subset|\OO(3)|$.  If we denote by $\cC_{\mathrm{fs}}$ the locus of singular
    points in the fibres over $U$, then as pointed out above, $\cC_{\mathrm{fs}}\to U$ is surjective. Let 
    \[
    \xymatrix{
        \cC^\nu\ar[rr]^\nu\ar[dr]&&\cC\ar[dl] \\
        &U&
    }
    \]
    be the normalisation of $\cC$. As a consequence of the proof of Lemma \ref{lem:Rnc}, we may shrink $U$ further and assume
    that $\cC^\nu$ is in fact smooth. The generic fibre of $\cC^\nu\to U$ is also smooth, so it can be checked from
    Zariski's main theorem that after shrinking $U$ further, $\nu_u:\cC^\nu_u\to\cC_u$ is the normalisation of every fibre over
    $u\in U$. Since $\nu^{-1}(\cC_{\mathrm{fs}})\to U$ has finite fibres, by purity of the branch locus we may shrink
    $U$ further and assume that every fibre of $\cC^\nu$ meets $\nu^{-1}(\cC_{\mathrm{fs}})$ transversely (at an even
    number of points). The map $\nu_u : \cC^\nu_u \to C_u$ identifies in pairs the above points forming nodes in the
    image because of the above mentioned  transversality condition. Hence, as $C$ is nodal and its arithmetic genus
    differs by 3780 from its normalisation, it must have precisely 3780 nodes. 
\end{proof}

\begin{corollary}
    The surface $V\subset F$ is irreducible and projective and has precisely 3780 isolated non-normal singularities,
    whose normalisation $\tilde{V}\to V$ is two-to-one over this non-normal locus. Hence $V$ is of general type.
\end{corollary}
\begin{proof}
    The first claims are from Corollary \ref{cor:classV}. Using notation around Equation \eqref{eq:Q01} and Lemma
    \ref{lem:tVandC}, $\tilde{V}$ meets a fibre $E_{S,[\ell]}\cong\PP^1$ in either one or two distinct points, namely
    the number of common roots of $Q_0, Q_1$ (noting that the case of a double root is ruled out by genericity). From
    Lemma \ref{lem:tVandC} and Theorem \ref{theo:nodesC}, there are 3780 triple lines of second type $\ell\in S\cap V$
    for which there are precisely two distinct 2-planes tangent to $\ell$ which have the same conjugate line, namely
    $\ell$. A local computation shows that $V$ cannot be singular along smooth points of $C$: over a smooth point of
    $C$, $\tilde{V}$ meets the fibre of the exceptional divisor $E_S$ at one point with multiplicity one and hence
    cannot have singular image. Hence the map $\tilde{V}\to V$ identifies these 3780 pairs of points giving as many
    non-normal isolated singularities as such pairs, and is an isomorphism otherwise.
\end{proof}

\begin{theorem}\label{thm:nodalcount}
    If $X$ is very general, every singular rational curve in $F$ of primitive class $\beta$ has precisely one node and
    there are 3780 such curves.
\end{theorem}
\begin{proof}
    From \cite[Proposition 6]{amerik}, the rational curves $\phi(E_{S,[\ell]})$ are of class $\beta$ in $\HH_2(F,\ZZ)$
    and Theorem \ref{thm:osy} proves that if $X$ is very general, every rational curve in $F$ of class $\beta$ is of
    this form. Next, \cite[Proposition 1.4, Lemma 3.8]{no} proves that every curve in $F$ of class $\beta$ is as worst
    nodal. If a fibre $E_{S,[\ell]}$ of $E_S$ over $[\ell]$ maps to a nodal curve under $\phi$, then the preimages of
    the node are two distinct points of $E_{S,[\ell]}$. These correspond to two distinct planes, tangent  along $\ell$
    and having the same residual line $\ell'$. But the only case in which this can happen is if $\ell=\ell'$ and these
    planes are triply tangent to the line $[\ell]$. This implies that $[\ell]\in C=S\cap V$ and the two planes
    correspond to the two distinct common roots of the quadrics $Q_0,Q_1$ of Equation \eqref{eq:Q01}. In this case, the
    image $\phi(E_{S,[\ell]})$ has exactly one node, at the point $[\ell]$. Alternatively, one could also argue that the
    images of these curves have at most one node as in \cite[\S 1.4]{no}, which implies that every curve of class
    $\beta$ is either a smooth rational curve or a plane cubic, hence is of arithmetic genus 0 or 1-nodal of arithmetic
    genus 1.

    In other words, singular curves in $F$ of class $\beta$ are all 1-nodal and rational, and are images of fibres of
    $E_S$ over the points of $C=S\cap V$ where the $Q_0, Q_1$ have two distinct common roots. As discussed above, these
    correspond exactly to the nodes of $C$ and so Theorem \ref{theo:nodesC} gives that there are 3780 of them.
\end{proof}

To make the geometry of the above count more explicit: each nodal rational curve of class $\beta\in\HH^2(F,\ZZ)$
corresponds to a singularity of the curve $C=S\cap V\subset F$, in the sense that it is the image of $\PP^1\cong
E_s\subset\tF$ for $s\in C_{\mathrm{sing}}$ under the Voisin map $\phi:\tF\to F$. In turn, each such singular point of
$C$ corresponds to a triple type two line which has two distinct triple tangent $2$-planes and these can be counted
intersection-theoretically.

\begin{remark}
    For general $X$, \cite[Lemmas 3.7-3.9]{no} prove that for $\Sigma\subset\CH_\beta(F)$ the locus of
    irreducible curves of arithmetic genus 1, every singular member is 1-nodal and that the morphism $\sigma:\Sigma\to\PP^1$
    taking a curve to its $j$-invariant is of degree 3780. In particular, the fibre over the boundary point at infinity
    corresponds to the number of 1-nodal rational curves in $F$ of class $\beta$, which as discussed in the proof above
    is the image of a line $E_{S,[\ell]}$ for $[\ell]\in S\cap V$. Theorem \ref{thm:nodalcount} thus proves that
    $\sigma$ is unramified at infinity.
\end{remark}

\section{Geometry of the Ramification Locus}\label{sec:ramloci}

As above, we work with a general cubic fourfold $X\subset\PP^5$. We begin with some basic facts stemming from the
geometric description of the ramification points of the $g^1_3$'s on the curves $C_x$ given in Proposition \ref{prop:ram
pts} and the discussion following that. As in the introduction, we define the following loci, recalling that $W'\subset
X$ was defined in \eqref{eq:W'} as the locus of points $x$ whose corresponding curve $C_x$ has worse than 1-nodal
singularities.

\begin{definition}\label{def:RN}
Denote by \[R, R'\subset I\] the closures of the locus of all ramification points and the residual to all ramification
points respectively, of one of the possibly two $g^1_3$ of the family of curves $p: I\to X$ restricted to $X\backslash
W'$. Denote also by $N$ the closure of the locus of triple ramification points.
\end{definition}

Note that the locus $R$ contains the locus $q^{-1}(S)$ as can be seen from the following. On a family of smooth
curves with a family of $g^1_3$'s, if we degenerate to a nodal curve, two of the ramification points ``converge" to the
node (as for example can be seen by the theory of admissible covers). A point $t\in R$ lying on the smooth part of the
curve $C_x$, with $x\in X\backslash W'$, is a ramification point of one of the (possibly) two $g^1_3$ that $C_x$
possesses. As explained at the end of Section \ref{sec:curve of lines}, if $t'\in C_x$ is the residual to $t$ in the
above $g^1_3$, then $x\in \ell_t\cap \ell_{t'}$, the 2-plane $\Pi $ spanned by $\ell_t$ and $\ell_{t'}$ is tangent to
$X$ along $\ell$ and then $\ell_{t'}$ is the line which is residual to $\ell_t$ in the intersection of $\Pi$  with $X$.
In particular, a point of $t\in N$ lying on a curve $C_x$ as above corresponds to a triple tangent plane $\Pi$ to $X$
along $\ell_t$. Note also that by the discussion following Corollary \ref{cor:2to1}, a node $t$ of $C_x$ corresponds to
a line $\ell_t$ of second type and (in general) a pair of tangent planes to $X$ along $\ell_t$ so that the two residual
lines to $\ell_t$ in the intersection of those planes with $X$ intersect $\ell_t$ at the point $x$. 

\begin{lemma} \label{lem:NV} 
For $q: I \to F$ the projection and $V\subset F$ the locus of triple lines, we have \[q^{-1}(V) = N \subset I.\] 
\end{lemma}
\begin{proof}
As the intersection of $q^{-1}(V)$ with the family of curves $p: I\to X$ restricted to $X\backslash W'$ is the triple
ramification locus from Proposition \ref{prop:ram pts}, we certainly have that $N\subseteq q^{-1}(V)$. We now prove the
reverse inclusion $N\supseteq q^{-1}(V)$. From Theorem \ref{thm:tVconn}, $V$ is irreducible. Hence, $q^{-1}(V)$ is not
contained in $N$ only if its intersection with $p^{-1}(W')$ contains an open set of $q^{-1}(V)$. But this cannot happen
because $V\neq S$ so the generic point $[\ell]\in V$ represents a line of first type, which in turn implies that
$p(q^{-1}[\ell])$ intersects $W'$ at most at finitely many points. 
\end{proof}

We recall that $I =\{ (p,\ell), p\in \ell \} \subset X \times F$ and that 
\[\tF=\{ (\ell, \Pi ), \Pi \mbox{ tangent to } X \mbox{ along } \ell \} \subset F \times G(3,6).\]
The pullback of the family $I$ from $F$ to $\tilde{F}$, under the
blow up $\pi:\tilde{F}\to F$ of $S$, is given by $\tilde{I}=\{ (p, (\ell, \Pi )), p \in \ell\} \subset X \times
\tilde{F}$. Let $\tilde{\pi}: \tilde{I} \to I$, with  $(p, (\ell, \Pi )) \to (p, \ell)$ be the natural
morphism. Note that by identifying $\tF$ with the graph of the Voisin map, we may represent the points of $\tilde{I}$
by tuples $(p, (\ell, \ell', \Pi)) $, with $(\p, (\ell, \Pi))$ as above and $\ell'$ the residual line to $\ell $ in the
intersection of $\Pi$ with $X$. Define now 
\[\tilde{R} = \{ (p, (\ell, \ell', \Pi ))\in \tilde{I}, \, p \in \ell \cap
\ell'\}\]
and consider the diagram  
\begin{align}\label{eq:tR}
\begin{split}
\xymatrix{
   \tilde{R}\ar@{^{(}->}[r] & \tilde{I}\ar[d]_{\tilde{\pi}} \ar[r]^-{\tilde{q}} & \tilde{F}\ar[d]^\pi \\
    R \ar@{^{(}->}[r] & I\ar[r]^q  & F.
}
\end{split}
\end{align}

\begin{theorem}\label{thm:Bl=R}
    We have that 
    \begin{enumerate}
        \item $\Bl_{\tilde{V}}\tilde{F}\cong \tilde{R}$, 
        \item $\tilde{\pi}(\tilde{R})\subset R$ and
        \item the induced map $\tilde{\pi}:  \tilde{R}\to R$ is generically two-to-one over $q^{-1}(S)$ and is
        an isomorphism otherwise. 
    \end{enumerate}
    As a consequence, $R$ is irreducible with normalisation $\tilde{R}$, and is smooth outside $q^{-1}(S)$ over
    which the normalisation is generically two-to-one.
\end{theorem}
\begin{proof}
    Continuing with the notation established in the previous section, denote by $\mathrm{pr}_1: \cG \to \G(2,6)$ the projection
    $(\ell,\ell',\Pi)\to \ell$. Define also \[\tilde{\cI}_\cG =\PP(\mathrm{pr}_1^*(\cU_2))\subset \PP^5 \times \cG\] and
    its subset
    \[\tilde{\cR}_\cG=\{ (p,(\ell,\ell',\Pi)),\, p\in \ell\cap \ell' \} \subset \tilde{\cI}_\cG.\]
    Any two fibres of $\tilde{\cR}_\cG\to\PP^5$ are isomorphic, and over a point $p\in \PP^5$, this is $\{
    (\ell,\ell',\Pi), p\in \ell\cap \ell' \} \subset \cG$ and this fibre is itself a $\PP^1$-fibration over
    $\{(\ell,\Pi), p\in \ell \subset \Pi\}$, parametrising lines in a plane passing through a point. Finally
    $\{(\ell,\Pi), p\in \ell \subset \Pi\}$ is again a $\PP^1$-fibration over the Schubert cycle $\Sigma_3^p\subset
    \G(3,6)$ parametrising 2-planes through the point $p\in \PP^5$. Note now that if we fix a hyperplane
    $H=\PP^4\subset\PP^5$ not containing $p$, we obtain an isomorphism $\Sigma_3^p\cong \G(2,H)=\G(2,5)$ by taking a
    2-plane $\Pi$ containing $p$ and giving the line $\Pi\cap H$. Hence the image in $\G(3,6)$ is smooth as it is
    isomorphic to a Grassmannian itself. All in all, we obtain that $\tilde{\cR}_\cG$ is smooth and irreducible.

    The projection $\tilde{\cR}_\cG \to \cG$ is a birational morphism with exceptional locus a $\PP^1$-fibration over
    $\cG_{\Delta}$. Since $\tilde{\cR}_\cG$ is smooth, it is necessarily isomorphic to the blow up of $\cG$ at
    $\cG_{\Delta}$ (see \cite[p.604, item 4]{gh}).

    We now pull everything back to $\PP^5\times \tilde{\cF}$ via the projection $\mathrm{pr}=(\id, \mathrm{pr}_2)$ for
    $\mathrm{pr}_2:\tilde{\cF}\to\cG$. In other words, denote by
    \[
    \tilde{\cI} = \mathrm{pr}^{-1}(\tilde{\cI}_\cG) \subset \PP^5\times\tilde{\cF}
    \]
    and $\tilde{\cR}\subset\tilde{\cI}$ the inverse image of $\tilde{\cR}_\cG$. The locus $\tilde{\cR}$ is smooth and
    irreducible as it admits a smooth morphism over $\tilde{\cR_\cG}$ which is smooth and irreducible as we showed
    above. Note that for $X=\VV(G)\in|\OO_{\PP^5}(3)|$ we have that the restrictions over $X$ are
    \begin{align*}
    \tilde{\cI}|_X &= \tilde{I}_X, \\
    \tilde{\cR}|_X &= \tilde{R}_X,
    \end{align*}
    where $\tilde{I}_X$ is as in Diagram \eqref{eq:tR}, i.e., is the blowup of the universal family of lines $I_X$ over
    $X$ at the universal locus of type two lines $q^{-1}(S_X)$, and $\tilde{R}_X$ a subvariety of it, both denoted with
    the subscript $X$ to distinguish between the fibres.

    As $\tilde{\cR}$ is smooth irreducible, the fibre $\tilde{R}_X=\tilde{\cR}|_X$ over a general point
    $X\in|\OO(3)|$ is smooth. It is also birational to $\tilde{F}_X$, and over the locus $\tilde{V}_X$ it is codimension
    one since it is a $\PP^1$-fibration over it. Now the irreducibility of $\tilde{\cR}$ implies that a general fibre
    $\tilde{R}_X$ is equidimensional, so from the aforementioned birationality it must be irreducible. As the fibre of
    $\tilde{R}_X$ over $\tilde{V}_X$ is a line $\PP^1$, and it is otherwise isomorphic to $\tilde{F}_X$, we obtain from
    \cite[p.604, item 4]{gh} that 
    \begin{align}\label{eq:RX=Bl}
        \tilde{R}_X \cong \Bl_{\tilde{V}_X}(\tilde{F}_X).
    \end{align}

    By construction \[\tilde{R}_X=\{(p, (\ell,\ell',\Pi)), p\in \ell\cap\ell'\} \subset X \times \tF_X.\] 
    By what we discussed after Proposition \ref{prop:ram pts} and Definition \ref{def:RN}, when $p\notin W'$ (i.e., the
    generic case), then $(p,(\ell,\ell',\Pi))$ maps to a general point of $R$. Hence from the irreducibility of
    $\tilde{R}_X$, we get a surjection $\tilde{\pi} : \tilde{R}_X \to R_X$. Over the complement of $q^{-1}(S_X)$ it is an
    isomorphism since there the map $\tilde{I}_X \to  I_X$ is an isomorphism. Over $q^{-1}(S_X)$ it is generically two to
    one from Corollary \ref{cor:2to1}.
\end{proof}

\begin{remark}
    In fact it is not hard to see, e.g., using \cite[Proposition IV-21]{eh} that the isomorphism of Equation
    \eqref{eq:RX=Bl} holds globally, i.e., that $\tilde{\cR}\cong\Bl_{\tilde{\cV}}(\tilde{\cF})$.
\end{remark}

By composing the blowup map with the projection to $X$, we obtain
\[\xymatrix{\Bl_{\tilde{V}}\tilde{F}\ar[r]&\tilde{I}\ar[r]& I\ar[r]& X.}\]
Outside the exceptional divisor, the first morphism sends a point $(\ell,\ell',\Pi)$, so that $X\cap\Pi=2\ell+\ell'$, to
the point in $\tilde{I}$ in the fibre over $[\ell]$ which is the point of intersection with $\ell'$. On the other hand,
it identifies the points in the exceptional divisor in the fibre over $(\ell,\ell,\Pi)$ with the points of $\ell$.

\begin{corollary}\label{cor:resolvepsitilde}
    The map $\psi:\tilde{F}\dashrightarrow X$ is resolved by blowing up $\tilde{V}$, and the resolution morphism is
    given by the above composition.
\end{corollary}

From the proposition above, the projection $q: R \to F$ is birational, thus $R$ is a birational model of $F$ inside $I$.

\begin{lemma}
    Let $X\subset\PP^5$ be a general cubic fourfold. The map $\psi: F \dashrightarrow X$ defined in Diagram
    \eqref{diag:defpsi} is a dominant rational map of degree $24$.
\end{lemma} 
\begin{proof}
    Let $x\in X$ be a general point. We want to count the number of pairs $(\ell,\ell')$ of distinct lines
    through $x$ such that the 2-plane spanned by $\ell,\ell'$ is tangent to $\ell$. From Proposition \ref{prop:ram pts}
    the number of such pairs is exactly the number of ramification points of the two $g^1_3$'s on $C_x$, which is $24$.
    Each contributes a distinct point as the ramification is simple at all points since the curve $C_x$ is
    general from Proposition \ref{prop:properties C_x}.
\end{proof} 

We now compute the classes of $R, R'$ and $N$ in the Chow group of $I$. Given Proposition \ref{prop:properties C_x},
perhaps the most natural way to do this is via admissible covers on the moduli space of curves, which we include as
Appendix \ref{sec:admissible covers} and is independent of the proof below (one uses Lemma \ref{lem:intersection theory}
to write these classes in terms of $q^*H_F, l$), but does not recover the class of $N$. 

\begin{proposition}\label{prop:classR}
    The classes of $R, R'$ and $N$ in $I$ are as follows, for $l=\mathrm{c}_1(\OO_I(1))$ the tautological class
    \begin{align*}
        [R]&=4q^*H_F+l,\\
        [R']&=4q^*H_F+16l,\\
        [N]&= 4l^2 -4lq^*H_F + 25q^*\mathrm{c}_2(\cU_F).
    \end{align*}
\end{proposition}

\begin{remark}
    We note that using the computation of the class $[R']$ as a consequence of Appendix \ref{sec:admissible covers} one
    recovers that the degree of the Voisin map satisfies
    \[\deg\phi = [R']q^*[\ell] = 16.\]
    In an earlier version of this paper we computed the classes of $R,R'$ and $N$ (and as a consequence the class of
    $V$) by avoiding the techniques of Section \ref{sec:surface V} and instead doing intersection theory on a partial
    resolution of singularities $J$ of $I\times_X I$, and in particular interpreting $R$ as the image of a natural
    degeneracy locus in $J$ - i.e., this method constructed a natural geometric correspondence over $(F,X)$ inside
    $I\times_X I$. This involved various technicalities to step around the locus $W'\subset X$
    parametrising the worse than 1-nodal fibres of $p:I\to X$, so in this version of the paper we chose instead to
    construct $\tilde{R}$ via the universal family in Proposition \ref{thm:Bl=R}, and then to resort to computations of
    \cite{amerik, svfourier, osy} for the following proof of Proposition \ref{prop:classR}.
\end{remark}

\begin{proof}[Proof of \ref{prop:classR}]
    The computation of the class of $N$ follows from Theorem \ref{thm:class of V} and Lemmas \ref{lem:NV},
    \ref{lem:intersection theory}. For the other two classes, we follow closely the notation of Theorem \ref{thm:class
    of V} and Proposition \ref{thm:Bl=R}. 

    The variety  $\cG=\{ (\ell,  \ell',\Pi) : \ell', \ell\subset\Pi \}  \subset
    \G(2,6)  \times \G(2,6) \times\G(3,6)$ maps to  ${\mathcal D} =\{ (\ell, \Pi), \ell \subset \Pi \}$  in two ways
    $\alpha_1$, $\alpha_2$, by mapping $(\ell, \ell',\Pi)$ to  $(\ell, \Pi)$ or to $(\ell',\Pi)$ respectively.  We also
    denote by $\alpha : \cG \to \G(3,6)$ the projection to the third factor.  We now consider the one-to-one morphism
    $\iota: \tF \to \cG$.  The inclusion $\gamma^*\cU_2\hookrightarrow \beta^*\cU_3$ in the exact sequence
    \eqref{eq:tautol_X} sitting over ${\mathcal D}$ yields, for $i=1,2$, inclusions
    $\iota^*\alpha_i^*\gamma^*\cU_2\hookrightarrow \iota^*\alpha^*\cU_3$ and thus an embedding
    \[j_i: P_i:=\PP (\iota^*\alpha_i^*\gamma^*\cU_2)\hookrightarrow P:=\PP(\iota^*\alpha^*\cU_3)\] 
    of bundles over $\tF$. Let $u:P\to \tF$ and $u_i=uj_i: P_i \to \tF$ be the natural maps.

    For $i,k\in\{1,2\}, i\neq k$, \cite[Proposition 9.13]{3264} gives 
    \begin{align}\label{eq:ji*}
    j_i^*j_{k*}[P_k] = \mathrm{c}_1({\mathcal O}_{P_i}(1))
    +u_i^*(\iota^*\alpha^*\mathrm{c}_1(\cU_3)-\iota^*\alpha_k^*\gamma^*\mathrm{c}_1(\cU_2))
    \end{align}
    on $P_i$.
  
    We now view  $\tF$ as the blow up of $F$ at $S$.  Under this point of view and recalling Diagram \eqref{eq:tR}, $P_1 =
    \tilde{I}:= \pi^*I$ and $P_2= \tilde{I}': =\phi^*I$.   We denote by $[\tilde{R}]=j_1^*j_{2*}[P_2]$ in $\tilde{I}$
    and by $[\tilde{R}']=j_2^*j_{1*}[P_1]$ in $\tilde{I}'$. The cycle $[\tilde{R}]$ in $\tilde{I}$ is geometrically the
    cycle representing the locus which on a  fibre $\ell$ of $\tilde{I}$ over a general point $(\ell, \Pi)$ of $\tF$ is
    the intersection point of $\ell$ with the residual line $\ell'$ in the plane $\Pi$. Thus,
    $\tilde{\pi}_*[\tilde{R}]=[R]$, with $\tilde{\pi}:\tilde{I} \to I$ the natural map corresponding to the blow up map
    $\pi: \tF \to F$.  Similarly, $\tilde{\phi}_*[\tilde{R}']=[R']$, with $\tilde{\phi}: \tilde{I}' \to I$ the natural
    map corresponding to the Voisin map $\phi: \tF \to F$. 
      
    We then have from \eqref{eq:ji*} and Lemma \ref{lem:huyamer}
    \begin{align*}
    [\tilde{R}]&=\mathrm{c}_1({\mathcal O}_{\tilde{I}}(1))+u_1^*\left((-3H+E_S)+(7H-3E_S)\right)\\
    &=\mathrm{c}_1({\mathcal O}_{\tilde{I}}(1))+u_1^*(4H-2E_S).
    \end{align*}
    By pushing forward to $I$ by the map $\tilde{\pi}$ we get $[R]=4q^*H_F+l$ as claimed.

    The computation for $[\tilde{R}']$ is similar. We have first from \eqref{eq:ji*}
    \begin{align*}
    [\tilde{R}'] &=\mathrm{c}_1({\mathcal O}_{\tilde{I}'}(1)) +u_2^*((-3H+E_S)+H)\\&=\mathrm{c}_1({\mathcal O}_{\tilde{I}'}(1))
    +u_2^*(-2H+E_S).
    \end{align*}
    By pushing forward to $I$ by the map $\tilde{\phi}$,  we get $\tilde{\phi}_*\mathrm{c}_1({\mathcal
    O}_{\tilde{I}'}(1))=16l$ as the degree of $v$ is 16. Finally, $\tilde{\phi}_*u_2^*H = 28q^* H_F$ (see
    \cite[Proposition 21.4]{svfourier}) and $\tilde{\phi }_*u_2^* E_S=60q^*H_F$ from \cite[Relation (18)]{osy}. We
    conclude $[R']=4q^*H_F+16l$.
\end{proof}

\appendix
\section{Admissible Covers}\label{sec:admissible covers} 
In this section we reinterpret the divisors $R, R'\subset I$ of Section \ref{sec:ramloci} using Hurwitz spaces, and
in particular obtain again some of the results of Proposition \ref{prop:classR}.

\begin{proposition}
\label{prop_classes}
Let $\lambda$ be the determinant of the Hodge bundle of the curve 
fibration $p: I \to X$. Then 
\begin{itemize}
    \item $[R]=4\omega_p+8p^*\lambda-p^*[W]$, 
    \item $[R']=4\omega_p+68p^*\lambda-8p^*[W]$.
\end{itemize}
\end{proposition}
\begin{proof}
    We denote by $\tilde{\mathcal{H}}_{4,3}$ the part of the space of admissible covers of genus $4$ and degree $3$,
    with $b=12$ labelled branched points, consisting of the union of the locus of smooth admissible covers and the
    boundary divisors $E_0$ and $E_3$ (see \cite{gk} for the notation) where a point of $E_0$ corresponds to a map $X\to P$,
    with $X=(C_1\cup R_1) \cup C_2$ and $P=\PP^1\cup \PP^1$, with $C_2$ a smooth genus $3$ curve and all
    other curves rational, as in \cite[Proposition 3.1, ii]{gk} for $k=2$. A point of $E_3$ corresponds to a map $X\to
    P$, with $X=S \cup C$, where $C$ is a smooth genus $4$ curve and all other curves rational, as in \cite[Proposition
    4.1, ii]{gk} for $k=2$. We have the diagram of universal curves
    \[
    \xymatrix{
    \overline{{\mathcal M}}_{0,b+1} \ar[d]_{\pi_b} & \mathcal{C}_{\tilde{\mathcal{H}}_{4,3}}
     \ar[d]^t \ar[l]_{f} \ar[r]^{\phi} & \overline{\mathcal M}_{4,1}
      \ar[d]^{\pi} \\
    \overline{\mathcal{M}}_{0,b} & \tilde{\mathcal{H}}_{4,3} \ar[l]_{j} \ar[r]^h & 
    \overline{\mathcal M}_4.
    }
    \]

    The map $h$ is generically finite of degree $N_0=2 \, b!$ which is the number of $g^1_3$'s times the possible
    permutations of the branch points. The map $t$ is equipped with $b$ disjoint sections $\tau_i:
    \tilde{\mathcal{H}}_{4,3} \to \mathcal{C}_{\tilde{\mathcal{H}}_{4,3}}$ corresponding to the $b$ labelled
    ramification points sitting over the labelled branched points. These commute with the natural sections $s_i:
    \overline{\mathcal{M}}_{0,b} \to \overline{\mathcal{M}}_{0,b+1}$, i.e., $f\tau_i= s_ij$. Let $T_i$ be the image of
    $\tau_i$. Note that each ramification point of a 3-to-1 map $F: X \to P$ admits all the labels - in other words,
    if $i \in \{1,...,b\}$ and $p$ is a ramification point of $F$ then there exists an $h\in
    \tilde{\mathcal{H}}_{4,3}$ so that the fibre over $h$ represents the map $F$ and $\tau_i(h)=p$. Therefore the
    images of the $T_i$'s under the map $\phi$ coincide and are equal to the union of all ramification points of all the
    $g^1_3$'s.

    The map $h$ is simply branched at $E_0$ and $E_3$. The map $\phi $ factors via the fibre product $\mathcal{X}=
    \overline{\mathcal{M}}_{4,1}\times _{\overline{\mathcal{M}}_4} \tilde{\mathcal{H}}_{4,3}$. Let 
    \[
    \xymatrix{
    \mathcal{C}_{\tilde{\mathcal{H}}_{4,3}} \ar[r]^\rho \ar[dr]^t & \mathcal{X} \ar[d]^{\tilde{\pi}} \ar[r]^{\tilde{h}} & 
    \overline{\mathcal M}_{4,1} \ar[d]^{\pi} \\
     & \tilde{\mathcal{H}}_{4,3} \ar[r]^h & \overline{\mathcal M}_4
    }
    \]
    the induced diagram, with $\phi= \tilde{h}\rho$. We have $\omega_{\tilde{\pi}}= \tilde{h}^* \omega_{\pi}$. The map
    $h$ is ramified at the points $x\in E_0$ and therefore the fibre of $\tilde{\pi}$ over $x$ has an $A_1$ singularity
    corresponding to the node over the fibre of $\pi$. The map $\rho$ resolves the singularity and the corresponding
    exceptional divisor $\mathcal{C}_1$ (the union of the $C_1$ from above) consists of $(-2)$-curves. In addition, on
    the fibres over $E_0$ the map $\rho$ contracts an exceptional divisor $\mathcal{R}_1$ which consists of
    $(-1)$-curves - blow-up along a smooth locus (the union of the tails $R_1$ as above). The map $h$ is also ramified
    at the points $x\in E_3$ but since the fibre over $x$ maps to a smooth curve, the fibre of $\tilde{\pi}$ over $x$ is
    a smooth curve. The map $\rho$ is the blow-up at the triple ramification point locus of the fibres of $\tilde{\pi}$.
    The exceptional divisor $\mathcal{S}$ consists of $(-1)$-curves (the curves $S$ as above). Since the restriction of
    $\omega_t$ on $(-2)$-curves is trivial as $C_1$ gets contracted via $\phi$, we get that $\omega_t
    =\phi^*\omega_{\pi} + B$ with $B=\mathcal{R}_1+ \mathcal{S}$.

    Let $\alpha$ be a divisor class on $\mathcal{C}_{\tilde{\mathcal{H}}_{4,3}}$. We then have 

    \begin{align*}
    \pi_*(\omega_{\pi} \phi_*\alpha) &= \pi_*\phi_*(\phi^* \omega_{\pi} \, \alpha) \\
    &= h_*t_*(\phi^* \omega_{\pi}\, \alpha)\\
    &= h_*t_*(\omega_t\, \alpha -B\alpha) \\
    &= h_*t_*(\omega_t\, \alpha) - h_*t_*(B\alpha). 
    \end{align*}
    We have $\omega_t=f^* \omega_{\pi_b} + R_f$, with $R_f$ the ramification divisor of the map $f$, that is
    $R_f=\sum_{i=1}^b T_i$. Then 

    \begin{align*}
     h_*t_*(\omega_t\, \alpha) &= h_*t_*(f^* \omega_{\pi_b}\,\alpha) + h_*t_*(R_i\,\alpha)\\
     &= h_*t_*(f^* \omega_{\pi_b}\,\alpha) +\sum_{i=1}^b h_*t_*(\tau_*\tau_i^*\alpha)\\
     &= h_*t_*(f^* \omega_{\pi_b}\,\alpha) +\sum_{i=1}^b h_*(\tau_i^*\alpha).
    \end{align*}

    Suppose now that $\alpha=R_f=\sum_{i=1}^bT_i$. Note that the degree of $R$ over $\overline{\mathcal{M}}_4$ is 2b so
    we have \[R={2\over N_0} \phi_*R_f.\] By \cite[Lemma 4.2]{gk2} we have
    \begin{alignat*}{2}
        \sum_{i=1}^b\tau_i^* R_f &= \sum_{i=1}^b \tau_i^*T_i &&= \sum_{i=1}^b t_*\tau_{i*}\tau_i^*T_i \\ 
        &= \sum_{i=1}^b t_*T_i^2 &&= t_*(R_f^2)=-{1\over 2} j^*\psi
    \end{alignat*}
    and $t_*(f^* \omega_{\pi_b}\,R_f) = j^*\psi$. Therefore $ h_*t_*(\omega_t\, R_f) = {1\over 2}h_* j^*\psi$. Also,
    $t_*(BR_f)=2E_3$. This is because $\mathcal{S}$ has two ramification points and $\mathcal{R}_1$ has none. Therefore
    \[
    \pi_*(\omega_{\pi}\, \phi_*R_f) = {1\over 2} h_*j^*\psi-2h_*E_3.
    \]

    Following formulas in \cite{gk} we have:
    \begin{align*}
        h_*j^* \psi &= 40 N_0 (9\lambda - \delta_0)\\
        2h_*E_3 &= N_0 (132\lambda - 15\delta_0)
    \end{align*}
    and so,
    \begin{align*}
    \pi_*(\omega_{\pi}\, R)& = {2\over N_0}\pi_*(\omega_{\pi}\, 
     \phi_*R_f)=40(9\lambda -\delta_0)-2(132\lambda -15\delta_0) \\ &=
     96\lambda -10\delta_0. 
    \end{align*}
    Writing $R=a \omega_{\pi}+b \pi^*\lambda+ c \pi^* \delta_0$, we get $a=4$ (because the fibre
    has 24 points and $24= 4 \cdot 6$, with $6=2g-2$). This implies
    \[
    \pi_*(\omega_{\pi}\, R)= 4(12\lambda -\delta_0) + 6 b \lambda+
    6 c \delta_0 = (48+6b)\lambda + (-4+6c)\delta_0. 
    \]
    Hence $48+6b= 96$ which implies $b=8$ and $-4+6c =-10$ giving $c =-1$.

    Suppose now that $\alpha=A$ is the sum of the residual sections to the $T_i$'s. Then $R'= {2\over N_0}\phi_*A$. We have
    $ h_*t_*(\omega_t\, A)= h_*t_*(f^* \omega_{\pi_b}\,A)=h_*t_*(f^*\omega_{\pi_b}A)= j^*\psi$ and $B=2E_0+2E_3$. This
    is because both $\mathcal{R}_1$ and $\mathcal{S}$ contain two points which are residual to the ramification points.
    Then 
    \[
    \pi_*(\omega_{\pi}\, \phi_*A) = h_*j^*\psi -2h_*E_0-2h_*E_3.
    \]
    We have $h_*E_0 = {N_0\over 2}\delta_0$. Therefore $\pi_*(\omega_{\pi}\, R_1)={2\over N_0}\pi_*(\omega_{\pi}\,
    \phi_*A)= 80(9\lambda -\delta_0)-2(132\lambda -15\delta_0)-2\delta_0 =456\lambda -52\delta_0 $. Writing now $R'=4
    \omega_{\pi}+b \pi^*\lambda+ c \pi^* \delta_0$ we have as above that $48+6b= 456$ which implies $b=68$
    and $-4+6c =-52$ giving $c =-8$.
\end{proof}

%\bibliographystyle{alpha}
%\bibliography{bib}

\end{document}